\documentclass[a4paper,11pt,leqno]{amsart}

\usepackage[dvips]{graphics}
\usepackage[matrix,arrow,tips,curve]{xy}
\usepackage[english]{babel}
\usepackage{amsmath}
\usepackage{amssymb}

\usepackage{mathrsfs}

\usepackage{enumerate}

\usepackage{psfrag}
 \psfrag{E}{\fontsize{20}{20}$E$}
 \psfrag{F}{\fontsize{20}{20}$F$}
 \psfrag{B}{\fontsize{20}{20}$\Delta_0$}
 \psfrag{l}{\fontsize{20}{20}$\ell$}
 \psfrag{E5}{\fontsize{20}{20}$\wi{E}$}
 \psfrag{E0}{\fontsize{20}{20}$E_0$}
 \psfrag{E1}{\fontsize{20}{20}$E_1$}
 \psfrag{E2}{\fontsize{20}{20}$E_2$}
 \psfrag{A1}{\fontsize{20}{20}$A_1$}
 \psfrag{A2}{\fontsize{20}{20}$A_2$}
 \psfrag{G1}{\fontsize{20}{20}$E_0'$}
 \psfrag{G}{\fontsize{20}{20}$E_0'$}
 \psfrag{G2}{\fontsize{20}{20}$E_0''$}
 \psfrag{D1}{\fontsize{20}{20}$D_1$}
 \psfrag{D}{\fontsize{20}{20}$D$}
 \psfrag{D2}{\fontsize{20}{20}$D_2$}
 \psfrag{R1}{\fontsize{20}{20}$R_1$}
 \psfrag{R2}{\fontsize{20}{20}$\w{R}_2$}
 \psfrag{R3}{\fontsize{20}{20}$\w{R}_3$}
 \psfrag{X3}{\fontsize{20}{20}$X_3$}
 \psfrag{R4}{\fontsize{20}{20}$\w{R}_4$}
 \psfrag{E3}{\fontsize{20}{20}$E_3$}
 \psfrag{E4}{\fontsize{20}{20}$E_4$}
 \psfrag{S3}{\fontsize{20}{20}$\sigma(E_3)$}
 \psfrag{S4}{\fontsize{20}{20}$\sigma(E_4)$}

\newtheorem{thm}{Theorem}[section]
\newtheorem{lemma}[thm]{Lemma}
\newtheorem{proposition}[thm]{Proposition}
\newtheorem{step}{Step}

\theoremstyle{definition}
\newtheorem{remark}[thm]{Remark}
\newtheorem{parg}[thm]{}

\newcommand{\ph}{\varphi}
\newcommand{\w}{\widetilde}
\newcommand{\la}{\longrightarrow}
\newcommand{\wi}{\widehat}
\newcommand{\pr}{\mathbb{P}}
\newcommand{\Q}{\mathbb{Q}}
\newcommand{\R}{\mathbb{R}}

\newcommand{\N}{\mathcal{N}_1}
\newcommand{\Nu}{\mathcal{N}^1}
\newcommand{\Sing}{\operatorname{Sing}}
\newcommand{\NE}{\operatorname{NE}}
\newcommand{\Exc}{\operatorname{Exc}}
\newcommand{\Lo}{\operatorname{Locus}}
\newcommand{\codim}{\operatorname{codim}}
\newcommand{\Eff}{\operatorname{Eff}}

\title[Fano manifolds with a rational fibration]{On some Fano manifolds admitting a rational fibration}
\author{C.~Casagrande}
\address{Universit\`a di Torino,
Dipartimento di Matematica,
via Carlo Alberto 10,
10123 Torino - Italy}
\email{cinzia.casagrande@unito.it}
\date{December 17, 2013}
\begin{document}
\maketitle
\section{Introduction}
Let $X$ be a smooth, complex Fano variety.
We denote by $\N(X)$ the real vector space of one-cycles on $X$, with real coefficients, modulo numerical equivalence; it dimension is the Picard number $\rho_X$ of $X$, which coincides with the second Betti number. 

\renewcommand{\theequation}{\thethm}

For any prime divisor $D\subset X$, let us consider the linear subspace $\N(D,X)$ of $\N(X)$ spanned by classes of curves contained in $D$, and define
$$c(D):=\codim \N(D,X)=\dim\ker\left(H^2(X,\R)\to H^2(D,\R)\right).$$
Notice that if $n\geq 3$ and $D$ is ample, then $c(D)=0$ by Lefschetz Theorem on hyperplane sections.
 
It is a special ``positivity property'' of Fano manifolds that for an arbitrary prime divisor $D$, $c(D)$ cannot be too large; moreover, the presence of a prime divisor $D\subset X$ with $c(D)>0$ has consequences on the geometry of $X$: the larger $c(D)$, the stronger these consequences.

More precisely, let us define
$$c_X:=\max\bigl\{c(D)\,|\,D\text{ a prime divisor in }X\bigr\}.$$
We have $c_X\in\{0,\dotsc,\rho_X-1\}$, and $c_X\geq \rho_X-\rho_D$ for every prime divisor $D\subset X$.
This invariant of $X$ has been introduced in \cite{codim}, and has some remarkable properties.
\begin{thm}[\cite{codim}, Th.~3.3]\label{codim}
Let $X$ be a smooth Fano  variety. We have
 $c_X\leq 8$, and moreover:
\begin{enumerate}[$\bullet$]
\item
if $c_X\geq 4$, then $X\cong S\times Y$, $S$ a Del Pezzo surface with $\rho_S=c_X+1$;
\item
if $c_X=3$, then there is an equidimensional %quasi-elementary\footnote{``Quasi-elementary'' means that every curve contracted by %$\psi$ is numerically equivalent to a one-cycle in the general fiber.} 
fibration in Del Pezzo surfaces\footnote{A fibration in Del Pezzo surfaces is a surjective morphism, with connected fibers, whose general fiber is a Del Pezzo surface.} 
$\psi\colon X\to Y$, where $Y$ is smooth and Fano, and $\rho_X-\rho_Y=4$.
\end{enumerate}
\end{thm}
In this paper we consider the next case, $c_X=2$. We show that 
up to a birational modification given by a sequence of flips, $X$ has either a conic bundle structure, or a fibration in Del Pezzo surfaces.
%Here is the precise statement.
\begin{thm}\label{main}
Let $X$ be a smooth Fano variety with
 $c_X=2$. Then one of the following holds:
\begin{enumerate}[$(i)$]
\item  there exist a sequence of flips $X\dasharrow X'$  and a conic bundle $f\colon X'\to Y$ where  $X'$ and $Y$ are smooth, $\rho_X-\rho_Y=2$, and $f$ factors through a 
 smooth $\pr^1$-fibration\footnote{A smooth morphism whose fibers are isomorphic to $\pr^1$.} over $Y$;
\item there is an equidimensional fibration in Del Pezzo surfaces $\psi\colon X\to Y$, where $Y$ is factorial, has  canonical singularities, $\codim\Sing(Y)\geq 3$, and $\rho_X-\rho_Y=3$.
\end{enumerate}
\end{thm}
We give a more detailed version of $(i)$ in Th.~\ref{main2}.

It is easy to find examples (among toric Fano varieties) of $X$ as in Th.~\ref{main}$(i)$ where a birational modification is necessary, namely $X$ itself does not have a conic bundle structure, nor a fibration in Del Pezzo surfaces.

On the other hand, the author knows no example of case $(ii)$ which does not also satisfy $(i)$ (just by taking a factorization of $\psi$ in elementary contractions, so that $X'=X$).

Results related to Th.~\ref{main} were already known in special cases, such as $\dim X\leq 4$, $\rho_X\leq 3$, or $X$ toric; we refer the reader to section \ref{last} for more details. We also refer the reader to \cite{gloria} for properties of $c_X$ in the singular case.

In particular, after Th.~\ref{codim} and Th.~\ref{main}, a Fano manifold $X$ with $c_X\geq 2$ is always covered by a family of rational curves of anticanonical degree $2$ (see also Rem.~\ref{X0}).

Finally, when $c_X=1$, we show that $X$ still has some (weaker) property, see Prop.~\ref{codimone}.

\medskip

Let us give an outline of the paper.
The proof of Th.~\ref{main} is based on a birational construction introduced in \cite{codim}, consisting in a ``Minimal Model Program'' (MMP) for $-D$, where $D\subset X$ is a prime divisor with $c(D)=c_X$ (see \S \ref{basic} for more details). The existence of such a MMP follows from \cite{BCHM}. This
construction has two outcomes: first, it yields some special smooth prime divisors in $X$, that we call ``exceptional $\pr^1$-bundles''\footnote{An exceptional  $\pr^1$-bundle is a  smooth prime divisor $E\subset X$, which is a $\pr^1$-bundle with fiber $e\subset E$, such that $E\cdot e=-1$.}; second, it gives a rational fibration on $X$ (which however a priori could be trivial, if the base turns out to be a point).

More precisely, we need to consider a ``special'' MMP, that is  a MMP for $-D$ where all extremal rays are $K$-negative.  This MMP can be of two types, which we call type $(a)$ and type $(b)$ (see \ref{basic}), and have different properties. The type of a MMP depends on the divisor $D$ and on the choice of the extremal rays of the MMP, and we have no control on it.
In view of Th.~\ref{main}, the case where
 the MMP is of type $(b)$ is better, as we automatically get a conic bundle defined on an open subset of $X$. 
All these properties, together with the results that we need from \cite{codim}, are recalled in section \ref{MMP}, where we also prove a few  related technical lemmas.

In section \ref{prel} we gather some other preliminary facts, mainly on exceptional $\pr^1$-bundles. The most important result here is Lemma \ref{4r}, which allows to describe (under suitable assumptions) extremal rays having positive intersection with an exceptional $\pr^1$-bundle $E$. This is used in the sequel  to study a MMP for $-E$.

The proof of Th.~\ref{main} is contained in sections \ref{secdelpezzo} and \ref{cb}. Let us sketch the strategy.

Let $X$ be a smooth Fano variety with $c_X=2$.
We first need to consider the possibility that for all choices of prime divisors $D\subset X$ with $c(D)=2$, every special MMP for $-D$ is of type $(a)$. The author has no explicit example of such a situation, but was not able to exclude it.
 This case is studied in  section \ref{secdelpezzo}, and is the hardest part of the paper.  We show (Th.~\ref{intermediate}) that either there is a $D$ as above and a special MMP of type $(b)$ for $-D$, or $X$ has a fibration in Del Pezzo surfaces as in Th.~\ref{main}$(ii)$.

Then in section \ref{cb} we describe the case of a special MMP of type $(b)$ (Prop.~\ref{conicbundle}), and we use Th.~\ref{intermediate} and 
Prop.~\ref{conicbundle} to prove
 Th.~\ref{main}.

Finally in section \ref{last} we study the case $c_X=1$ and consider in more detail some special cases, in particular the toric case.
{\small\tableofcontents}
\vspace{-30pt}
\subsection*{Acknowledgements}
I am grateful to St\'ephane Druel for many useful 
conversations on this work. I also thank the referee for some useful suggestions. The 
author was partially supported by the Research Project M.I.U.R. PRIN 2009 ``Spazi di moduli e teoria
di Lie''.
\stepcounter{thm}
\subsection{Notations}\label{notation}
 We will use the definitions and apply the techniques of the Minimal Model
Program frequently, without explicit references. We refer the reader to
\cite{kollarmori} for terminology and details.

Let $X$ be a normal and $\Q$-factorial projective variety, of dimension $n$.
We denote by $\equiv$ numerical equivalence (for both curves and divisors), and by $\Nu(X)$ the real vector space of Cartier divisors in $X$, with real coefficients, modulo numerical equivalence. For every divisor $D$ or curve $C$ on $X$, we denote by $[D]\in\Nu(X)$ and $[C]\in\N(X)$ the respective numerical equivalence classes. We also set $D^{\perp}:=\{\gamma\in\N(X)\,|\,D\cdot \gamma=0\}\subset\N(X)$.
If $R$ is an extremal ray of $X$, we write 
$D\cdot R>0$ (or $D\cdot R<0$) if $D\cdot\gamma>0$ (respectively, $D\cdot\gamma<0$)
for $\gamma\in R$ non-zero.  We say that $R$ is \emph{$K$-negative} if $K_X\cdot R<0$.
We denote by $\Lo(R)\subseteq X$ the union of curves $C\subseteq X$ such that $[C]\in R$.

 A \emph{contraction} of $X$ is a surjective morphism $\ph\colon X\to Y$ with connected fibers, where $Y$ is normal. We set $\NE(\ph):=(\ker\ph_*)\cap\NE(X)$, where
$\NE(X)\subset\N(X)$ is the convex cone generated by classes of effective curves.
We say that $\ph$ is \emph{$K$-negative} if $K_X\cdot C<0$ for every curve $C\subset X$ such that $\ph(C)=\{pt\}$.

If $\ph$ is of fiber type, we say that $\ph$ is
{\bf quasi-elementary} if every curve contracted by $\ph$ is numerically equivalent to a one-cycle in a general fiber, see \cite[\S 3]{fanos}. In particular, an elementary contraction of fiber type is always quasi-elementary.

We say that $R$ is {\bf an extremal ray of type $(n-1,n-2)^{sm}$} if its contraction $\ph\colon X\to Y$ is the blow-up of a smooth, irreducible subvariety of codimension two in $Y$, contained in the smooth locus of $Y$. 

\medskip

For every closed subset $Z\subset X$, we denote by $\N(Z,X)$ the subspace of $\N(X)$ generated by classes of curves in $Z$, namely $\N(Z,X)=i_*(\N(Z))$, where $i\colon Z \hookrightarrow X$ is the inclusion.

\medskip

A \emph{smooth $\pr^1$-fibration} is a smooth morphism whose fibers are isomorphic to $\pr^1$. A \emph{$\pr^1$-bundle} is the projectivization of a rank two vector bundle.
A \emph{conic bundle} is a proper morphism between smooth varieties, such that every fiber is isomorphic to a plane conic, see \cite[\S 4]{wisn}.

\medskip

Let $X$ be a smooth Fano variety. By \cite[Cor.~1.3.2]{BCHM}, $X$ is a Mori dream space. In particular, this implies that we can freely use all tools of the Minimal Model Program, see \cite[Prop.~1.11]{hukeel}. Moreover, the convex cone $\Eff(X)\subset\Nu(X)$, generated by classes of effective divisors, is closed and polyhedral, see \cite[Prop.~1.11(2)]{hukeel}.

A \emph{contracting birational map} is a birational map $f\colon X\dasharrow X'$ such that $X'$ is projective, normal, and $\Q$-factorial, and $f^{-1}$ does not contract divisors. Equivalently, $f$ factors as a finite sequence of flips and elementary divisorial contractions.

We say that a prime divisor $D\subset X$ is \emph{fixed} if it
coincides with its stable base locus, namely if it is not movable. This happens, for instance, if $D$
 is covered by a family of curves with which it has negative intersection. 
One can see that $D$ is fixed if and only if, after a sequence of flips, it becomes the exceptional divisor of a divisorial contraction (see for instance \cite[Rem.~2.19]{eff}). Moreover, if $D$ is fixed, then $[D]$ spans a one-dimensional face of the cone of effective divisors $\Eff(X)$, and $D$ is the unique prime divisor having class in this face.

An {\bf exceptional $\pr^1$-bundle} is a smooth prime divisor $E\subset X$, which is a $\pr^1$-bundle with fiber $e\subset E$, such that $E\cdot e=-1$. In particular, $E$ is a fixed prime divisor.
Throughout the paper we will adopt the following {\bf convention:} 

\emph{An exceptional $\pr^1$-bundle is always denoted with a capital letter $E$, $F$, $G$, etc. When working with the divisor $E$, we fix a $\pr^1$-bundle structure on $E$. A fiber of this $\pr^1$-bundle is always denoted with the lower-case letter corresponding to the divisor, e.g.\ $e\subset E$, $f\subset F$, $e_1\subset E_1$, $\wi{e}\subset \wi{E}$, and so on. }
\section{Preliminaries on special MMPs}\label{MMP}
\stepcounter{thm}
\subsection{Special MMPs and exceptional $\pr^1$-bundles}\label{basic}
In this paragraph we recall from \cite[\S 2]{codim} the construction that we will use to prove Th.~\ref{main}.

Let $X$ be a smooth Fano variety, and  consider a prime divisor $D\subset X$ with $c(D)>0$.
A {\bf special MMP for $-D$} is a sequence
$$\xymatrix{
{X=X_1}\ar@/^1pc/@{-->}[rrrr]^{\sigma}\ar@/_1pc/@{-->}[drrrr]_{\psi}
\ar@{-->}[r]_{\,\quad \sigma_1}& 
{X_2}\ar@{-->}[r]&
{\cdots}\ar@{-->}[r] &
{X_{k-1}}\ar@{-->}[r]_{\sigma_{k-1}}& 
{X_k}\ar[d]^{\ph}\\
&&&&Y 
}$$
where:
\begin{enumerate}[$\bullet$]
\item every $X_i$ is a normal and $\Q$-factorial projective variety, with terminal singularities;
\item for every $i=1,\dotsc,k-1$ there exists an extremal ray $R_i$ of $X_i$ such that $D_i\cdot R_i>0$ ($D_i\subset X_i$ being the transform of $D\subset X$) and $-K_{X_i}\cdot R_i>0$;
\item $\ph\colon X_k\to Y$ is a $K$-negative elementary contraction of fiber type, such that $\ph(D_k)=Y$.
\end{enumerate}
We set $\sigma:=\sigma_{k-1}\circ\cdots\circ\sigma_1$ and $\psi:=\ph\circ\sigma$.

The word ``special'' refers to a choice of a MMP for $-D$ where all involved extremal rays have positive intersection with the anticanonical divisor. This is possible because $X$ is Fano, using a MMP with scaling of $-K_X$ \cite[Rem.~3.10.9]{BCHM} (see \cite[Prop.~2.4]{codim} for a proof in this specific context).

This construction is studied in detail in \cite[\S 2]{codim}, and by \cite[Lemma 2.7(2)]{codim}
there are two possible cases, depending on $\N(D_k,X_k)\subseteq\N(X_k)$:
\begin{enumerate}[$(a)$]
\item $c(D_k)=0$ ({\bf MMP of type $(a)$});
\item $c(D_k)=1$ ({\bf MMP of type $(b)$}).
\end{enumerate}
We describe separately the two types (when $c(D)=c_X=2$) in paragraphs \ref{typea} and \ref{typeb}; let us first see a technical property that we will need in the sequel.
\begin{lemma}\label{flips}
Let $i\in\{2,\dotsc,k\}$ and let
$\Gamma$ be a one-cycle in $X$, with real coefficients,  such that its transform\footnote{More precisely, $\Gamma_{j+1}$ is defined as $(\sigma_j)_*\Gamma_j$ if $\sigma_j$ is a divisorial contraction, and as the transform of $\Gamma_j$ if $\sigma_j$ is a flip.} $\Gamma_j$ in $X_j$ has  no components contained in 
  the indeterminacy locus of $\sigma_j$, for every $j=1,\dotsc,i-1$.

Let $G\subset X$ be a prime divisor, and define $G_i$ to be the
 transform of $G$ in $X_i$ if $G$ is not contracted by $X\dasharrow X_i$, and $G_i=0$ otherwise.

If $[\Gamma]\in\N(G,X)+\N(D,X)$, then $[\Gamma_i]\in\N(G_i,X_i)+\N(D_i,X_i)$, 
where we set $\N(G_i,X_i):=\{0\}$ if $G_i=0$.
\end{lemma}
\begin{proof}
By induction on $i$, we can assume that $[\Gamma_{i-1}]\in\N(G_{i-1},X_{i-1})+\N(D_{i-1},X_{i-1})$. Consider $\sigma_{i-1}\colon X_{i-1}\dasharrow X_i$. 

Suppose that $\sigma_{i-1}$ is a divisorial contraction. We have
$(\sigma_{i-1})_*(\N(D_{i-1},X_{i-1}))=\N(D_i,X_i)$, which settles the case $G_{i-1}=0$. If instead $G_{i-1}$ is a prime divisor, we have
$(\sigma_{i-1})_*(\N(G_{i-1},X_{i-1}))=\N(\sigma_{i-1}(G_{i-1}),X_i)$, so if $G_{i-1}$ is not exceptional, then $\sigma_{i-1}(G_{i-1})=G_i$ and the statement is clear. 
Finally if $G_{i-1}=\Exc(\sigma_{i-1})$, then  we have $\sigma_{i-1}(G_{i-1})\subset D_i$
because $D_{i-1}\cdot R_{i-1}>0$, and this gives again the statement.

If $\sigma_{i-1}$ is a flip, let us consider the diagram associated with the flip:
$$\xymatrix{{X_{i-1}}\ar@{-->}[rr]^{\sigma_{i-1}}\ar[dr]_{\zeta}&&{X_i}\ar[dl]^{\zeta'}\\
&Z&
}$$
We consider the case where $G_{i-1}$ is a prime divisor, the case $G_{i-1}=0$ being similar. We have $\zeta_*([\Gamma_{i-1}])=\zeta'_*([\Gamma_i])$, $\zeta(G_{i-1})=\zeta'(G_i)$, and  $\zeta(D_{i-1})=\zeta'(D_i)$, hence
$$\zeta'_*([\Gamma_i])\in \N(\zeta(G_{i-1}),Z)+\N(\zeta(D_{i-1}),Z)
=\zeta'_* \bigl(\N(G_{i},X_i)+\N(D_{i},X_i)\bigr).$$
Since $D_{i-1}\cdot\NE(\zeta)>0$, we have $D_i\cdot\NE(\zeta')<0$, hence $\ker\zeta'_*\subset\N(D_i,X_i)$, which yields the statement.
\end{proof}
\stepcounter{thm}
\subsection{Properties of a special MMP of type $(a)$, when $c_X=c(D)=2$.}\label{typea}
All the statements in this paragraph (except the last line)
follow from \cite[Lemma 2.7]{codim}.

There are two special indices $i_1,i_2\in\{1,\dotsc,k-1\}$,  $i_1<i_2$, such that $R_{i_j}\not\subset\N(D_{i_j},X_{i_j})$ for $j=1,2$, and $R_i\subset\N(D_i,X_i)$ for every $i\in\{1,\dotsc,k\}\smallsetminus\{i_1,i_2\}$ (in particular $k\geq 3$).

For $j=1,2$, $\sigma_{i_j}\colon X_{i_j}\to X_{i_j+1}$ is the blow-up of a smooth subvariety of codimension $2$, contained in the smooth locus of $X_{i_j+1}$.
 Moreover, $\Exc(\sigma_{i_j})$ is contained in the open subset where the birational map $X_{i_j}\dasharrow X$ is an isomorphism.

Let $E_1$ and $E_2$ be the transforms in $X$ of the exceptional divisors of $\sigma_{i_1}$ and $\sigma_{i_2}$ respectively. Then $E_j$ is an exceptional $\pr^1$-bundle such that $D\cdot e_j>0$, for $j=1,2$; the divisors $D$, $E_1$ and $E_2$ are distinct, and $E_1\cap E_2=\emptyset$ (see fig.~\ref{ta}).

Finally by \cite[Lemma 3.1.8]{codim} we have $c(E_j)=2$ and $\dim\N(D\cap E_j,X)=\rho_X-3$ for $j=1,2$. 
\stepcounter{thm}
\begin{figure}[h]
\begin{center} 
\scalebox{0.3}{\includegraphics{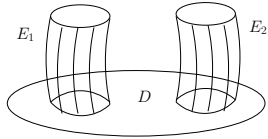}}
\caption{MMP of type $(a)$.}
\label{ta}
\end{center}
\end{figure}
\stepcounter{thm}
\subsection{Properties of a special MMP of type $(b)$, when $c_X=c(D)=2$.}
\label{typeb} 
All the statements in this paragraph (except Lemma \ref{min}) follow from
\cite[Lemmas 2.7 and 2.8]{codim}.

We have $\ker\ph_*\not\subset\N(D_k,X_k)$, and
 $\ph\colon X_k\to Y$ is finite on $D_k$, so that every fiber of $\ph$ has dimension one, and $\dim Y=n-1$. The general fiber of $\ph$ is a smooth rational curve.

Let $T\subset X_k$ be the indeterminacy locus of $\sigma^{-1}$. 
We have:
 $$\Sing(X_k)\subseteq T\subset D_k,$$
and every fiber of $\ph$ intersecting $T$ is an integral rational curve.

There is a special index $i_1\in\{1,\dotsc,k-1\}$ such that $R_{i_1}\not\subset\N(D_{i_1},X_{i_1})$, and $R_i\subset\N(D_i,X_i)$ for every $i\in\{1,\dotsc,k-1\}\smallsetminus\{i_1\}$. 

The map $\sigma_{i_1}\colon X_{i_1}\to X_{i_1+1}$ is the blow-up of a smooth subvariety of codimension $2$, contained in the smooth locus of $X_{i_1+1}$. 
Moreover $\Exc(\sigma_{i_1})$ (respectively, $\sigma_{i_1}(\Exc(\sigma_{i_1}))$) is contained in the open subset where 
the  birational map $X_{i_1}\dasharrow X$ (respectively,
 $X_{i_1+1}\dasharrow X_k$) is an isomorphism. 

Let $E\subset X$ be the transform of $\Exc(\sigma_{i_1})$, so that $E$ is an exceptional $\pr^1$-bundle in $X$ such that $D\cdot e>0$, and $\sigma\colon 
X\dasharrow X_k$ is regular around $E$.

Set $A:=\sigma(E)\subset X_k$. Then $A$ is smooth of dimension $n-2$, 
 $$A\subseteq T\subset D_k,$$ and $A$ is a connected component of $T$. 
The morphism $\ph$ is finite on $A$, and the image $Z:=\ph(A)\subset Y$ is a prime divisor.

Consider the divisor $\ph^{-1}(Z)\subset X_k$, and let $\widehat{E}\subset X$ be its transform.
Then $E\cup \widehat{E}\subset X$
 (respectively, $\ph^{-1}(Z)$) is contained in the open subset where the birational map  $X\dasharrow X_{i_1}$
(respectively,  $X_k\dasharrow X_{i_1+1}$) is an isomorphism. Thus $\psi\colon X\dasharrow Y$ is a regular conic bundle in a neighborhood of  $E\cup \widehat{E}$, $Z=\psi(E)=\psi(\wi{E})$, and
 $\psi^*(Z)=E+\wi{E}$  (see fig.~\ref{tb}).

Let $\ell\subset X$ be a general fiber of $\psi$ (and we still denote by $\ell$ its transform in each $X_i$, so that $\ell\subset X_k$ is a general fiber of $\ph$). 
The divisor $\wi{E}$ is an exceptional $\pr^1$-bundle, $\ell\cap (E\cup\wi{E})=\emptyset$, and we have:
 $$[\wi{e}\,]\not\in\N(E,X), \quad
\ell\equiv e+\wi{e},\quad
E\cdot\ell=\wi{E}\cdot\ell=0,\quad E\cdot\wi{e}=\wi{E}\cdot e=1.$$ 
\stepcounter{thm}
\begin{figure}[h]
\begin{center} 
\scalebox{0.3}{\includegraphics{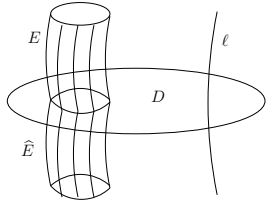}}
\caption{MMP of type $(b)$.}
\label{tb}
\end{center}
\end{figure}
\begin{lemma}\label{min}
Under the assumptions of \ref{typeb}, both $X_k$ and $Y$ are smooth, $\ph\colon X_k\to Y$ is a conic bundle, and $\ph$ is smooth on $\ph(T)$.
\end{lemma}
\begin{proof}
 We reproduce an argument from 
 \cite[3.3.7]{minimal}.
Since every fiber of $\ph$ intersecting $T$ is an integral rational curve, by \cite[Th.~II.2.8]{kollar}
$\ph$ is smooth in a neighbourhood of $\ph^{-1}(\ph(T))$. Thus
$\Sing(X_k)=\ph^{-1}(\Sing(Y)\cap\ph(T))$, because
 $\Sing(X_k)\subseteq T$. This implies that 
$\Sing(X_k)=\emptyset$, as $\ph$ is finite on $T$. Therefore $Y$ is smooth  and $\ph$ is a conic bundle
(see \cite[Th.~4.1(2)]{AWaview} and references therein). 
\end{proof}
\begin{parg}
We conclude this section with some additional properties needed in the sequel.
We keep the same notation as in the previous paragraphs.
\begin{lemma}\label{directsum0}
Let $X$ be a smooth Fano variety with $c_X=2$, $D\subset X$ a prime divisor with $c(D)=2$, and consider a special MMP for $-D$.

If the MMP is of type $(a)$, we have:
$$\N(X)=\N(D,X)\oplus\R[e_1]\oplus\R[e_2].$$

If the MMP is of type $(b)$, we have:
$$\N(X)=\N(D,X)\oplus\R[e]\oplus\R[\wi{e}\,].$$
\end{lemma}
\begin{proof}
The proofs in the two cases are very similar; we consider the case where the MMP is of type $(a)$. We keep the same notation as in \ref{typea}.

We have $\R[e_1]\neq\R[e_2]$ (for instance because $E_1\cdot e_1=-1$ and $E_1\cdot e_2=0$), so it is enough to show that $\N(D,X)\cap (\R[e_1]\oplus\R[e_2])=\{0\}$.

Suppose that $\lambda_1[e_1]+\lambda_2[e_2]\in\N(D,X)$ for some $\lambda_i\in\R$, set $\Gamma:=\lambda_1e_1+\lambda_2e_2$, and consider the birational map $X\dasharrow X_{i_2}$. 

For every $j<i_1$,  $e_1$ and $e_2$ are contained in the open subset where $\sigma_j$ is an isomorphism, so the transform of $\Gamma$ in $X_{i_1}$ is $\Gamma_{i_1}=\lambda_1e_1+\lambda_2e_2$ (for simplicity we still denote by $e_1$ and $e_2$ their transforms along the MMP). 

The blow-up $\sigma_{i_1}$ contracts $e_1$ to a point and $e_2\cap\Exc(\sigma_{i_1})=\emptyset$, so that  the transform of $\Gamma$ in $X_{i_1+1}$ is $\Gamma_{i_1+1}=(\sigma_{i_1})_*(\Gamma_{i_1})=\lambda_2e_2$.

Finally, for every $j<i_2$, $e_2$ is contained in the open subset where $\sigma_j$ is an isomorphism, and  the transform of $\Gamma$ in $X_{i_2}$ is $\Gamma_{i_2}=\lambda_2e_2$. 

Now applying Lemma \ref{flips} (with $G=D$), we deduce that 
 $\lambda_{1}[e_{1}]+\lambda_{2}[e_{2}]\in\N(D_{i_1},X_{i_1})$
and
$\lambda_2 [e_2]\in\N(D_{i_2},X_{i_2})$.

We have $R_{i_2}\not\subset\N(D_{i_2},X_{i_2})$ and $[e_{2}]\in R_{i_2}$, hence  $[e_2]\not\in\N(D_{i_2},X_{i_2})$, and we deduce that
 $\lambda_2=0$ and  $\lambda_{1}[e_{1}]\in\N(D_{i_1},X_{i_1})$. Similarly 
$[e_1]\not\in\N(D_{i_1},X_{i_1})$ yields
 $\lambda_1=0$, and we get the statement.
\end{proof}
\begin{remark}[\cite{codim}, Rem.~3.1.3]\label{elem2}
Let $X$ be a smooth projective variety,
$E\subset X$ an exceptional $\pr^1$-bundle, and $D\subset X$ a prime divisor with $D\cdot
e>0$. Then the following holds:
\begin{enumerate}[$(1)$]
\item $\N(E,X)=\R[e]+\N(D\cap E,X)$;
\item
either $[e]\in\N(D\cap E,X)$ and $\N(D\cap E,X)=\N(E,X)$, 
or $[e]\not\in\N(D\cap E,X)$ and $\N(D\cap E,X)$ has codimension
  $1$ in $\N(E,X)$; 
\item
for every irreducible curve $C\subset E$ we have
$C\equiv \lambda
e + \mu \w{C}$,
 where $\w{C}$ is a 
curve contained in $D\cap E$, $\lambda,\mu\in\R$, and
$\mu\geq 0$.
\end{enumerate}
\end{remark}
\begin{lemma}\label{directsum}
Let $X$ be a smooth Fano variety with $c_X=2$, $D\subset X$ a prime divisor with $c(D)=2$, and consider a special MMP of type $(a)$ for $-D$.
 We have: 
$$\N(E_1,X)\cap\N(E_2,X)=\N(D\cap E_1,X)\cap\N(D\cap E_2,X).$$
\end{lemma}
\begin{proof}
Clearly $\N(D\cap E_1,X)\cap\N(D\cap E_2,X)\subseteq \N(E_1,X)\cap\N(E_2,X)$.

Conversely, let $\gamma\in\N(E_1,X)\cap\N(E_2,X)$, and fix $i\in\{1,2\}$.
Since $D\cdot e_i>0$, we have $\N(E_i,X)=\R[e_i]+\N(D\cap E_i,X)$ by Rem.~\ref{elem2}(1).
Therefore we can write
 $$\gamma=\lambda_i[e_i]+\eta_i$$ 
with $\lambda_i\in\R$ and $\eta_i\in\N(D\cap E_i,X)\subseteq\N(D,X)$ for $i=1,2$.  By Lemma \ref{directsum0}, we deduce that $\lambda_1=\lambda_2=0$, and $\gamma=\eta_1=\eta_2\in\N(D\cap E_1,X)
\cap\N(D\cap E_2,X)$. 
\end{proof}
\end{parg}
\stepcounter{thm}
\section{Preliminary results on exceptional $\pr^1$-bundles}\label{prel}
In this section we show some properties that we will need in the sequel, mainly concerning  exceptional $\pr^1$-bundles and extremal rays. 
\begin{remark}\label{intersection}
Let $X$ be a smooth Fano variety with $c_X=2$, $D\subset X$ a prime divisor, and $B_1,B_2,B_3\subset X$ distinct fixed prime divisors. Then $D$ must intersect at least one of the $B_i$'s.

Indeed, the classes $[B_1],[B_2],[B_3]\in\Nu(X)$ generate distinct one-dimensional faces of the effective cone $\Eff(X)$ (see \ref{notation}); in particular, these classes are linearly independent. This implies that the intersection $L:=B_1^{\perp}\cap B_2^{\perp}\cap B_3^{\perp}$ in $\N(X)$ has codimension $3$. As $c(D)\leq c_X=2$, we cannot have $\N(D,X)\subseteq L$, thus $D$ cannot be disjoint from $B_1\cup B_2\cup B_3$.
\end{remark}
Similarly one shows the following.
\begin{remark}\label{inter2}
Let $X$ be a smooth Fano variety with $c_X=2$, $D\subset X$ a prime divisor, and $B_1,B_2\subset X$ distinct fixed prime divisors. 
If $D\cap (B_1\cup B_2)=\emptyset$, then
$$c(D)=2\quad\text{and}\quad\N(D,X)=B_1^{\perp}\cap B_2^{\perp}.$$
\end{remark}
\begin{remark}\label{e}
Let $X$ be a smooth projective variety,
$E\subset X$ an exceptional $\pr^1$-bundle, and $D_1,D_2\subset X$ two disjoint prime divisors such that $D_i\cap E\neq\emptyset$ for $i=1,2$. Then one of the following holds:
\begin{enumerate}[$(i)$]
\item $D_1\cdot e>0$ and $D_2\cdot e>0$ (see fig.~\ref{figlemma1});
\item  $D_1\cdot e=D_2\cdot e=0$, and $[e]\in\N(D_1,X)\cap\N(D_2,X)$.
\end{enumerate}

Indeed, $D_1$ and $D_2$ must be distinct from $E$, so that $D_i\cdot e\geq 0$ for $i=1,2$.
Suppose that $D_1\cdot e=0$. Since $D_1$ intersects $E$, then $D_1$ must contain some curve $e$; in particular $[e]\in\N(D_1,X)$. As $D_1\cap D_2=\emptyset$, this also gives $D_2\cdot e=0$, and finally  $[e]\in\N(D_2,X)$.
\end{remark}
\begin{lemma}\label{1}
Let $X$ be a smooth Fano variety and $E\subset X$ an exceptional 
  $\pr^1$-bundle. Suppose  that there exist two disjoint prime divisors $D_1,D_2\subset X$ such that $D_i\cdot e>0$ for $i=1,2$ (see fig.~\ref{figlemma1}).

Let  $R$ be
an  extremal ray of $X$ whose contraction is not finite on $E$, and such that $[e]\not\in R$. Then we have one of the possibilities:
\begin{enumerate}[$(a)$]
\item $D_1\cdot R<0$, $D_2\cdot R=0$, $\Lo(R)\subseteq D_1$;
\item $D_1\cdot R=0$, $D_2\cdot R<0$, $\Lo(R)\subseteq D_2$;
\item $D_1\cdot R=D_2\cdot R=0$, and there are curves with class in $R$ both in $D_1\cap E$ and in $D_2\cap E$.
\end{enumerate}
In particular $D_i\cdot R\leq 0$ for $i=1,2$, and the contraction of $R$ is not finite on $D_1\cup D_2$.
\stepcounter{thm}
\begin{figure}[h]
\begin{center} 
\scalebox{0.3}{\includegraphics{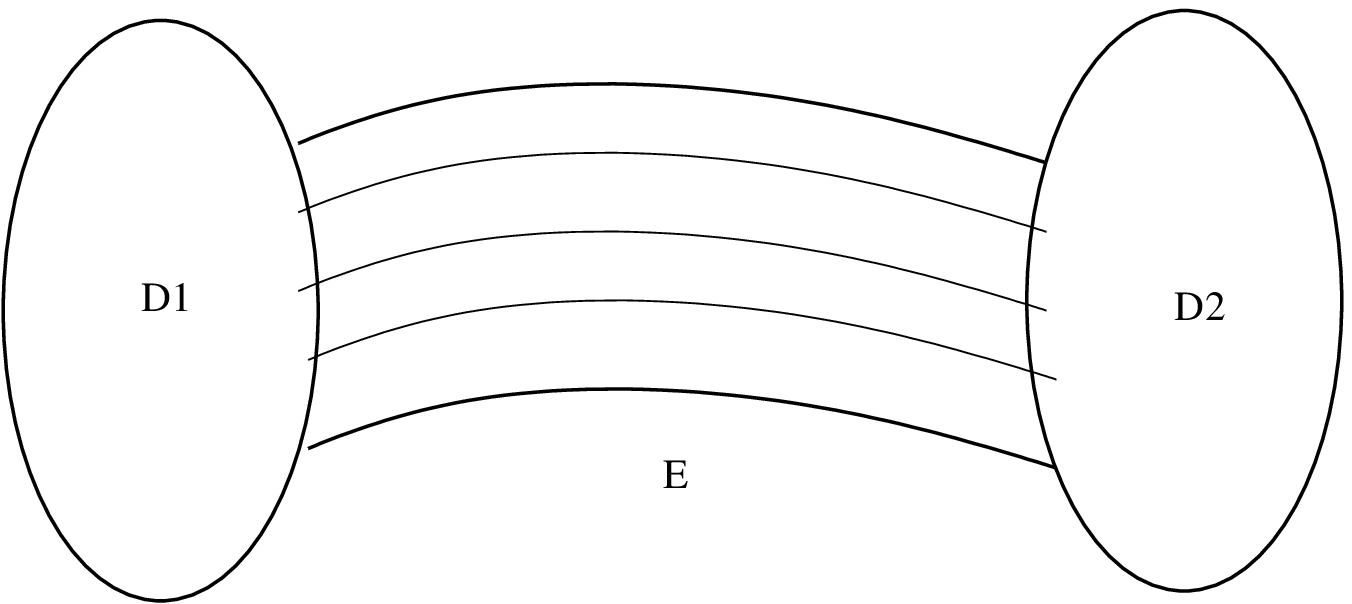}}
\caption{}
\label{figlemma1}
\end{center}
\end{figure}
\end{lemma}
\begin{proof}
Let $C\subset E$ be an irreducible curve such that $[C]\in R$. Since $D_2\cdot e>0$, by Rem.~\ref{elem2}(3) we have
$$C\equiv \lambda e+\mu\w{C}$$
where $\lambda,\mu\in\R$, $\mu>0$, and $\w{C}$ is a curve in $E\cap D_2$. As $R$ is an extremal ray, $[e]\not\in R$, and $\mu>0$, we must have $\lambda\leq 0$.

Intersecting with $D_1$ we get $D_1\cdot C=\lambda D_1\cdot e\leq 0$, hence $D_1\cdot R\leq 0$. If $D_1\cdot R<0$, we get $(a)$. If $D_1\cdot R=0$, then $\lambda=0$ and hence $[\w{C}]\in R$. 

Applying the same reasoning to $D_2$, we get $(b)$ or $(c)$.
\end{proof}
\begin{lemma}\label{4r}
Let $X$ be a smooth Fano variety with $c_X=2$, and let $E_0,E_1,E_2\subset X$ be exceptional 
 $\pr^1$-bundles  such that
 $c(E_i)=2$ for $i=0,1,2$. Suppose that $E_1\cap E_2=\emptyset$, and that $E_0\cdot e_i>0$ and $E_i\cdot e_0>0$ for $i=1,2$. Finally assume that there exist prime divisors $D_1,D_2$ (possibly equal) such that $D_i\cap E_i\neq\emptyset$  and $D_i\cap E_0=\emptyset$ for $i=1,2$ (see fig.~\ref{figlemma4r}).

Let $R$ be an extremal ray of $X$ such that $E_0\cdot R>0$. 

Then $R$ is  of type $(n-1,n-2)^{sm}$, $\Lo(R)\cap (E_1\cup E_2)\neq\emptyset$,
the contraction of $R$ is finite on $E_0$, and for $i\in\{1,2\}$ 
 the contraction of $R$ is finite on $E_i$ unless $[e_i]\in R$.
\stepcounter{thm}
\begin{figure}[h]
\begin{center} 
\scalebox{0.3}{\includegraphics{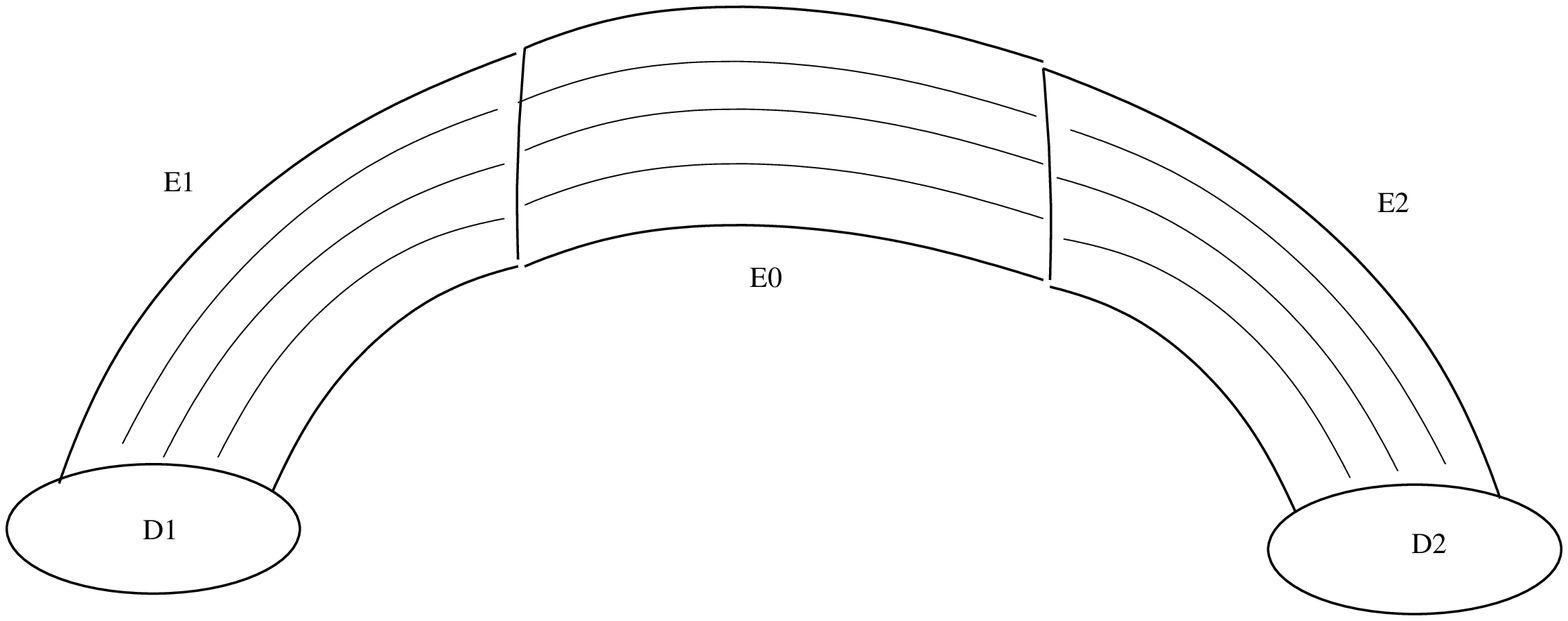}}
\caption{The divisors in Lemma \ref{4r}.}\label{figlemma4r}
\end{center}
\end{figure}
\end{lemma}
\begin{proof}
We observe first of all that $R$ cannot be of fiber type. This is because if the contraction $\ph\colon X\to Y$ of $R$ is of fiber type, then $\ph(E_0)=Y$ (because $E_0\cdot R>0$) and hence $\ph_*(\N(E_0,X))=\N(Y)$, which is impossible as $\dim\N(Y)=\rho_X-1$ and $\dim\N(E_0,X)=\rho_X-2$. Thus the contraction of $R$ is birational.

We show that $[e_1]$ and $[e_2]$ do not belong to $\N(E_0,X)$: let $i\in\{1,2\}$. One can easily see that $\dim\N(E_0\cap E_i,X)=\rho_X-3$ (see \cite[Lemma 3.1.8]{codim}), and hence $[e_i]\not\in\N(E_0\cap E_i,X)$ by Rem.~\ref{elem2}(2).
We have 
$$\N(E_0\cap E_i,X)\subseteq\N(E_0,X)\cap \N(E_i,X),$$
and $\dim\N(E_0,X)=\dim\N(E_i,X)=\rho_X-2$.
On the other hand $\N(E_0,X)\neq \N(E_i,X)$ (for instance because
 $\N(E_i,X)\subset E_{3-i}^{\perp}$ while $\N(E_0,X)\not\subset E_{3-i}^{\perp}$), thus $\N(E_0\cap E_i,X)=\N(E_0,X)\cap \N(E_i,X)$. Since $[e_i]\not\in\N(E_0\cap E_i,X)$
and $[e_i]\in\N(E_i,X)$, we conclude that
$[e_i]\not\in\N(E_0,X)$.

\medskip

Suppose that 
 $[e_i]\in R$ for some $i\in\{1,2\}$. Since $[e_i]\not\in\N(E_0,X)$, we have $R\not\subset\N(E_0,X)$, hence the contraction of $R$ must be finite on $E_0$. On the other hand, every non-trivial fiber $F$ of the contraction must intersect $E_0$, because $E_0\cdot R>0$. This yields $\dim F=1$, hence
 $R$ is of type $(n-1,n-2)^{sm}$ by \cite[Th.~1.2]{wisn}, and we have the statement.

\medskip

Let us assume now that $[e_1]$ and $[e_2]$ do not belong to $R$. 

We show that the contraction of $R$ is finite on $E_1\cup E_2$. By contradiction, suppose that
 the contraction of $R$ is not finite on $E_i$, with $i\in\{1,2\}$.
Then
 we apply Lemma \ref{1} to $E_i$ and get $E_0\cdot R\leq 0$, against our assumptions (here we use the existence of the divisor $D_i$). Thus the contraction of $R$ is finite on $E_1\cup E_2$. 

Notice that since $E_0\cdot R>0$ and $E_0\cdot e_0=-1$, we also have $[e_0]\not\in R$. Using again Lemma \ref{1} on $E_0$, we conclude that the contraction of $R$ is finite on $E_0$ too. As before, this implies that
 $R$ is of type $(n-1,n-2)^{sm}$.

We are left to show that $E_R\cap (E_1\cup E_2)\neq\emptyset$, where $E_R:=\Lo(R)$. By contradiction, assume  that $E_R\cap(E_1\cup E_2)=\emptyset$. 

Then by \cite[Lemmas 3.1.8 and 3.1.7]{codim} we deduce that there exists a linear subspace $L\subset\N(X)$, \emph{of codimension $3$,} such that: 
$$L\subseteq \N(E_0,X)\cap E_1^{\perp}\cap E_2^{\perp}\cap E_R^{\perp}
\subseteq  D_1^{\perp}\cap E_1^{\perp}\cap E_2^{\perp}\cap E_R^{\perp}$$
(recall that $E_0\cap D_1=\emptyset$, hence $\N(E_0,X)\subset D_1^{\perp}$).

This implies that $[D_1],[E_1],[E_2],[E_R]$ are linearly dependent in $\Nu(X)$, and there exist $\lambda_i\in\Q$, not all zero, such that
$$\lambda_0D_1+\lambda_1E_1+\lambda_2E_2+\lambda_3E_R\equiv 0.$$
Intersecting with $e_j$ ($j=1,2$) and $e_R$, we get $\lambda_j=\lambda_0D_1\cdot e_j$ and  $\lambda_3=\lambda_0D_1\cdot e_R$, so that $\lambda_0\neq 0$, and this yields:
$$D_1+(D_1\cdot e_1)E_1+(D_1\cdot e_2)E_2+(D_1\cdot e_R)E_R\equiv 0.$$
Notice that $D_1$ is distinct from $E_1$, $E_2$, and $E_R$, because these three divisors all intersect $E_0$, while $D_1\cap E_0=\emptyset$. Thus the coefficients above are non-negative;
but this is impossible, because a non-zero effective divisor cannot be numerically trivial. Therefore $E_R\cap(E_1\cup E_2)\neq \emptyset$.
\end{proof}
\begin{lemma}\label{indet}
Let $X$ be a $\Q$-factorial projective variety with Gorenstein terminal singularities, $\phi\colon X\dasharrow X'$ a sequence of flips of $K$-negative extremal rays, and $L\subset X'$ the indeterminacy locus of $\phi^{-1}$. Then $\codim L\geq 3$. 
\end{lemma}
\begin{proof}
Let $\Lambda\subset X'$ be an irreducible closed subset of codimension $2$. Notice that $X'$ has terminal singularities, in particular $\codim\Sing(X')\geq 3$, so that $X'$ is smooth at the general point of $\Lambda$. Let $\mu\colon X''\to X'$ be the blow-up of $\Lambda$, and let $D\subset X''$ be the unique irreducible component of $\Exc(\mu)$ which dominates $\Lambda$. Then $D$ is a divisor over $X'$, with center $\Lambda$ on $X'$,  and discrepancy $a(D,X')=1$ (see \cite[Def.~2.24 and Def.~2.25]{kollarmori}). By \cite[Lemma 3.38]{kollarmori} we have $a(D,X)\leq a(D,X')$. On the other hand, since $X$ has Gorenstein terminal singularities, it has integral and positive discrepancies, so that  $a(D,X)\geq 1$ and hence  $a(D,X)= a(D,X')$. Again by \cite[Lemma 3.38]{kollarmori}, this implies that $\Lambda$  cannot be contained in $L$. This shows that $\codim L\geq 3$. 
\end{proof}
\begin{proposition}\label{small}
Let $X$ be a smooth Fano variety, $X'$ a normal and $\Q$-factorial projective variety, $f\colon X\dasharrow X'$ a contracting birational map, and $L\subset X'$ the indeterminacy locus of $f^{-1}$. Then we have the following: 
\begin{enumerate}[$(i)$]
\item
$f$ can be factored as a sequence of divisorial contractions or flips of $K$-negative extremal rays;
\item
$X'$ has terminal singularities, $\Sing(X')\subseteq L$, and for every irreducible curve $\Gamma\subset X'$ not contained in $L$  we have $-K_{X'}\cdot \Gamma \geq 1$, with strict inequality if $\Gamma\cap L\neq \emptyset$;
\item if  $\ph\colon X'\to Y$ is an elementary contraction which is not $K$-negative, then $\ph$ is small and $\Exc(\ph)\subseteq L$;
\item
if $\ph\colon X'\to Y$ is an elementary birational contraction with fibers of dimension at most one, then $\Exc(\ph)\cap L$ is a union of fibers of $\ph$.
If moreover $\ph$ is small, then  $\Exc(\ph)\subseteq L$.
\end{enumerate}
\end{proposition}
\begin{proof}
Suppose first that $(i)$ holds.
Then $X'$ has terminal singularities, $\Sing(X')\subseteq L$, and if $\Gamma$ is as in $(ii)$ and $\w{\Gamma}\subset X$ is its transform, we have:
$$
-K_{X'}\cdot \Gamma\geq -K_X\cdot\w{\Gamma}\geq 1,\text{ and $-K_{X'}\cdot \Gamma> 1$ if $\Gamma\cap L\neq\emptyset$.}
$$
This follows from
\cite[Lemma 3.38]{kollarmori}. So $(ii)$ and $(iii)$ hold.

\medskip

To show $(i)$, we proceed by induction on $\rho_X-\rho_{X'}$. 

Suppose first that $\rho_X=\rho_{X'}$, so that  $f$ is an isomorphism in codimension one. Let $A$ be an ample Cartier divisor in $X'$, and consider the divisor $f^*(A)$ in $X$ (defined as follows: consider  $f^*(A)$ in the open subset where $f$ is regular, and then take the closure in $X$). Then $f^*(A)$ is a movable divisor in $X$, and any MMP for $f^*(A)$ yields a factorization of $f$ as a sequence of flips. Since $X$ is Fano, we can use 
a MMP with scaling of $-K_X$ (see \cite[Rem.~3.10.9]{BCHM} and \cite[Prop.~2.4]{codim}). This gives a factorization of $f$ as a sequence of flips of $K$-negative extremal rays.

Suppose now that  $\rho_X>\rho_{X'}$, and consider any  factorization of $f$ as a  (finite) sequence of elementary divisorial contractions and flips.  
Let $\sigma\colon Z\to Z'$ be 
 the last divisorial contraction occurring in the sequence; then $f$ factors as
  $X\dasharrow Z \stackrel{g}{\dasharrow}X'$, where $g$ is a contracting rational map given by $\sigma$  followed possibly by some flips. 

Notice that $Z$ is a normal and $\Q$-factorial projective variety, and it is a Mori dream space (see \cite[Prop.~1.11(2)]{hukeel}). By the properties of Mori dream spaces (see \emph{e.g.} \cite[\S  2.1 and references therein]{eff}), we can factor $g$ as $Z\stackrel{h}{\dasharrow} T \stackrel{\sigma'}{\to}X'$, where $T$ is a normal and $\Q$-factorial projective variety, $h$ is given by a finite sequence of flips, and $\sigma'$ is a birational morphism.

We remark that $\sigma'$ is elementary, because $\rho_T=\rho_Z=\rho_{Z'}+1=\rho_{X'}+1$. Moreover $\sigma'$ cannot be small because $X'$ is $\Q$-factorial (see \cite[Cor.~2.63]{kollarmori}), so that $\sigma'$ is an elementary divisorial contraction. 

In the end we get a factorization of $f$ as:
$$\xymatrix{
{X}\ar@/^1pc/@{-->}[rr]^{f}
\ar@{-->}[r]_{f'}& 
{T}\ar[r]_{\sigma'}&{X'}
}$$
where  $f'$ is a contracting birational map.
By induction, $f'$ factors as  a sequence of divisorial contractions or flips of $K$-negative extremal rays. Moreover, by $(iii)$, $\sigma'$ is $K$-negative too. This yields $(i)$.

\medskip

We show $(iv)$. If $\ph$ is not $K$-negative, we have $\Exc(\ph)\subseteq L$ by $(iii)$.
Let us assume that $\ph$ is $K$-negative.

Let $C\subset X'$ be an irreducible component of a 
fiber of $\ph$. We show that either $C\cap L=\emptyset$, or $C\subseteq L$.

Indeed, suppose that $C$ intersects $L$. If $C\subseteq\Sing(X')$, then $C$ is contained in $L$. If instead  $C\not\subseteq\Sing(X')$, then $-K_{X'}\cdot C\leq 1$, see \cite[\S 2, (2.3.2)]{mori} and \cite[Lemma 1(ii)]{ishii}.
This implies again that $C\subseteq L$ by $(ii)$.

Now let $F$ be a non-trivial fiber of $\ph$. Since $F$ is connected, we have that either $F\cap L=\emptyset$, or $F\subseteq L$. This shows that $\Exc(\ph)\cap L$ is a union of fibers of $\ph$.

Finally suppose that $\Exc(\ph)\not\subseteq L$. We set $Y_0:=Y\smallsetminus\ph(L)$, $X_0:=\ph^{-1}(Y_0)$, and $\ph_0:=\ph_{|X_0}\colon X_0\to Y_0$. Then $X_0$ is smooth and $\ph_0$ is a  birational, $K$-negative contraction, with fibers of dimension at most one. Then $\ph_0$ is divisorial (see \cite[Th.~4.1 and references therein]{AWaview}), and the same holds for $\ph$.
\end{proof}
\section{Constructing a fibration in  Del Pezzo surfaces}\label{secdelpezzo}
In this section we show the following result, which is an intermediate step 
 towards Th.~\ref{main}.
\begin{thm}\label{intermediate}
Let $X$ be a smooth Fano variety with $c_X=2$.
Then one of the following holds:
\begin{enumerate}[$(i)$]
\item  there exist a prime divisor $D\subset X$ with $c(D)=2$ and a special MMP of type $(b)$ for $-D$;
\item there is an equidimensional, quasi-elementary\footnote{See \ref{notation} for the definition of quasi-elementary.}
 fibration in Del Pezzo surfaces $X\to Y$, where $Y$ is factorial, has  canonical singularities, $\codim\Sing(Y)\geq 3$, and $\rho_X-\rho_Y=3$.
\end{enumerate}
\end{thm}
Let us outline the strategy of the proof. Suppose that $(i)$ does not hold.
Then the construction explained in paragraphs \ref{basic} and \ref{typea} yields a machinery that produces lots of pairs of disjoint exceptional $\pr^1$-bundles. After Lemma \ref{4r}, if some exceptional $\pr^1$-bundles intersect in a convenient way, we have a good description of extremal rays having positive intersection with these divisors. 
Thus the strategy is to produce a bunch of exceptional $\pr^1$-bundles having a good configuration (the reader can look at fig.~\ref{fig2} and \ref{fig1} to get an idea); this is achieved in steps \ref{4or5} and \ref{step1}. Then, in step \ref{delpezzo},  we use Lemma \ref{4r} to describe explicitly a special MMP for $-E_0$, where $E_0$ is one of these exceptional $\pr^1$-bundles. We show that the MMP has no flips and only two divisorial contractions; in the end this special MMP yields a fibration in Del Pezzo surfaces.
\begin{proof}[Proof of Th.~\ref{intermediate}]
We suppose  that $(i)$ does not hold, and show $(ii)$. This means that throughout the proof we assume the following:
\stepcounter{thm}
\begin{equation}\label{noa}
\text{\em for any prime divisor $D\subset X$ with $c(D)=2$, every special MMP for $-D$ is of type $(a)$.}
\end{equation}

So let $D\subset X$ be a prime divisor with $c(D)=2$. We consider a special MMP of type $(a)$ for $-D$, and this yields
 two disjoint exceptional $\pr^1$-bundles $E_0$ and $E_0'$ in $X$, such that $c(E_0)=c(E_0')=2$ (see \ref{typea}).

Now we repeat the construction with $E_0$: we consider a special MMP for $-E_0$, and get a new pair of disjoint exceptional $\pr^1$-bundles $E_1$ and $E_2$, such that $E_0\cdot e_1>0$ and $E_0\cdot e_2>0$, and $c(E_1)=c(E_2)=2$. Notice that $E_0,E_0',E_1,E_2$ are distinct.

Since $E_0\cap E'_0=\emptyset$, by Rem.~\ref{intersection} $E_0'$ must intersect  at least one of $E_1$, $E_2$. 

Moreover, since $E_1$ and $E_2$ are disjoint and both intersect $E_0$, we have two possibilities (see Rem.~\ref{e}): either $E_0\cap E_1$ and $E_0\cap E_2$ are both horizontal for the $\pr^1$-bundle on $E_0$ (that is: $E_1\cdot e_0>0$ and $E_2\cdot e_0>0$), or  $E_0\cap E_1$ and $E_0\cap E_2$ are both union of fibers $e_0$ (namely  $E_1\cdot e_0=E_2\cdot e_0=0$).
\begin{step}\label{4or5}
Up to replacing $E_0,E_0'$ with another pair,
we have one of the following cases:
\begin{enumerate}[$(i)$]
\item 
$E_0'$ intersects both $E_1$ and $E_2$;
\item  $E_0'\cap E_1\neq\emptyset$, $E_0'\cap E_2=\emptyset$, and  there exists  an exceptional $\pr^1$-bundle $E_0''\subset X$ such that: 
\stepcounter{thm}
\begin{equation}\label{listdisj}
c(E_0'')=2,\quad (E_0\cup E_1)\cap E_0''=\emptyset,\quad E_0''\cap E_2\neq\emptyset,\quad
E_0''\cap E_0'\neq\emptyset.\end{equation}
Moreover, whenever two among the five $\pr^1$-bundles $E_0,E_1,E_2,E_0',E_0''$ intersect each other, each divisor has positive intersection with the fiber of the $\pr^1$-bundle on the other divisor
(see fig.~\ref{fig2}).
\stepcounter{thm}
\begin{figure}[h]
\begin{center} 
\scalebox{0.3}{\includegraphics{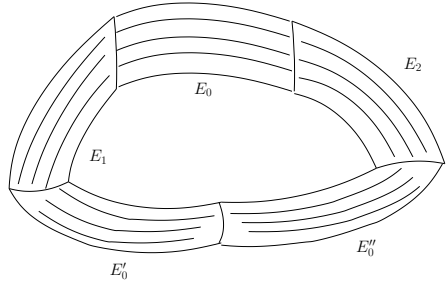}}
\caption{The case of $5$ exceptional $\pr^1$-bundles (step \ref{4or5}$(ii)$).}\label{fig2}
\end{center}
\end{figure}
\end{enumerate}
\end{step}
\begin{proof}[Proof of step \ref{4or5}]
We consider a special MMP for $-E_0'$, and we get
 disjoint exceptional $\pr^1$-bundles $F_1,F_2$ such that $E_0'\cdot f_1>0$, $E_0'\cdot f_2>0$, and $c(F_1)=c(F_2)=2$ (see \eqref{noa} and \ref{typea}).

If $E_0'$ intersects both $E_1$ and $E_2$, we have $(i)$. If $E_0$ intersects both $F_1$ and $F_2$, we just exchange $E_0$ and $E_0'$ and get again $(i)$.

Otherwise, suppose that $E_0'\cap E_2=\emptyset$ and $E_0\cap F_2=\emptyset$.
Then $E_0'$ is disjoint from $E_0$ and $E_2$, hence
$E_0'\cap E_1\neq\emptyset$  by Rem.~\ref{intersection}; similarly $E_0\cap F_1\neq\emptyset$.

If $F_2$ intersects both $E_1$ and $E_2$, we replace $E_0'$ by $F_2$ and we have again $(i)$. 

Finally, suppose that $F_2$ is disjoint from $E_1$ or $E_2$. Since
$F_2$ is already disjoint from the divisors $F_1$ and $E_0$, and $E_0$ is distinct from $E_1$ and $E_2$, by Rem.~\ref{intersection}
$F_1$ must coincide with $E_1$ or $E_2$. As  $E_0'\cap E_2=\emptyset$ and  $E_0'\cap F_1\neq\emptyset$, we must have $E_1=F_1$ and $F_2\cap E_2\neq\emptyset$.

We set   $E_0'':=F_2$. Then  we have \eqref{listdisj}, and by construction  $E_0\cdot e_1>0$, $E_0\cdot e_2>0$, and $E_0'\cdot e_0''>0$.  

Since $E_0\cap E_0'=\emptyset$, $E_0\cdot e_1>0$, and $E_0'\cap E_1\neq\emptyset$, we must have $E_0'\cdot e_1>0$ by Rem.~\ref{e}.
Similarly we see that $E_0''\cdot e_2>0$ and $E_2\cdot e_0''>0$.

We run a special MMP for $-E_1$. This yields two disjoint exceptional $\pr^1$-bundles $G_1,G_2$ such that $E_1\cdot g_i>0$ for $i=1,2$ (see \eqref{noa} and \ref{typea}).  

If $E_0''$ or $E_2$ intersect both $G_1$ and $G_2$, we replace $E_0$ by $E_1$, and we have $(i)$. 
Otherwise, suppose that $E_0''\cap G_1=\emptyset$. Since $E_0''$ is already disjoint from $E_0$ and $E_1$, and $G_1\neq E_1$,
by Rem.~\ref{intersection} we must have $G_1=E_0$. 
We have $E_2\cap E_0\neq\emptyset$, and $E_2$ does not intersect  both $G_1$ and $G_2$, therefore $E_2\cap G_2=\emptyset$, and again by Rem.~\ref{intersection} we get $G_2=E_0'$.
 
This shows that up to replacing the $\pr^1$-bundle structures on $E_0$ and $E_0'$ with other ones, we can assume that $E_1\cdot e_0>0$ and $E_1\cdot e_0'>0$.

By Rem.~\ref{e}, this implies also that $E_2\cdot e_0>0$ and $E_0''\cdot e_0'>0$, therefore we have $(ii)$.
\end{proof}
\begin{step}\label{step1} Suppose that $E_0'$ intersects both $E_1$ and $E_2$ (case $(i)$ in step \ref{4or5}).
Then,  whenever two among the four $\pr^1$-bundles $E_0,E_1,E_2,E_0'$ intersect each other, each divisor has positive intersection with the fiber of the $\pr^1$-bundle on the other divisor (see fig.~\ref{fig1}).
\stepcounter{thm}
\begin{figure}[h]
\begin{center} 
\scalebox{0.3}{\includegraphics{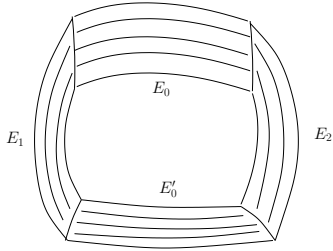}}
\caption{The case where $E_0'$ intersects both $E_1$ and $E_2$.}\label{fig1}
\end{center}
\end{figure}
\end{step}
\begin{proof}[Proof of step \ref{step1}]
By construction we have $E_0\cdot e_1>0$ and $E_0\cdot e_2>0$.
Since $E_0\cap E_0'=\emptyset$ and $E_0\cdot e_1>0$, by Rem.~\ref{e} we must also have  $E'_0\cdot e_1>0$, and similarly $E_0'\cdot e_2>0$.

By Lemma \ref{directsum} we have:
\stepcounter{thm}
\begin{equation}\label{eq0}
\N(E_1,X)\cap\N(E_2,X)=\N(E_0\cap E_1,X)\cap\N(E_0\cap E_2,X).
\end{equation}
On the other hand, as $E_0\cap E_0'=\emptyset$, 
we also have $\N(E_0,X)\subset (E_0')^{\perp}$, and we conclude that:
$$
\N(E_1,X)\cap\N(E_2,X)\subset (E_0')^{\perp}.
$$

As $E_0'\cdot e_0'=-1$, the class $[e_0']$ cannot belong to $(E_0')^{\perp}$, and we deduce that
$[e_0']\not\in\N(E_1,X)\cap\N(E_2,X)$. Then  Rem.~\ref{e} implies that  $E_1\cdot e'_0>0$ and $E_2\cdot e'_0>0$.

We are left to show that  $E_1\cdot e_0>0$ and $E_2\cdot e_0>0$. By contradiction, suppose otherwise. Then, 
again by Rem.~\ref{e}, we have
 $E_1\cdot e_0=E_2\cdot e_0=0$ and 
$$
[e_0]\in\N(E_1,X)\cap\N(E_2,X).
$$

We have $\N(E_i,X)\supseteq\N(E_0'\cap E_i,X)$ for $i=1,2$, hence
$$\N(E_1,X)\cap\N(E_2,X)\supseteq\N(E_0'\cap E_1,X)\cap\N(E_0'\cap E_2,X).$$
On the other hand:
$$\N(E_0'\cap E_1,X)\cap\N(E_0'\cap E_2,X)\subseteq E_0^{\perp},$$
because $E_0\cap E_0'=\emptyset$. As $E_0\cdot e_0=-1$, we see that 
 $[e_0]$ cannot be contained in $\N(E_0'\cap E_1,X)\cap\N(E_0'\cap E_2,X)$, and hence there is a \emph{strict} inclusion:
$$\N(E_1,X)\cap\N(E_2,X)\supsetneq\N(E_0'\cap E_1,X)\cap\N(E_0'\cap E_2,X).$$

By \cite[Lemma 3.1.8]{codim}, the subspaces  $\N(E_0'\cap E_1,X)$ and $\N(E_0'\cap E_2,X)$ have both dimension $\rho_X-3$, and are contained in $\N(E_0',X)$ which has dimension $\rho_X-2$. Thus their intersection
$\N(E_0'\cap E_1,X)\cap\N(E_0'\cap E_2,X)$ has dimension at least $\rho_X-4$, 
and this yields
$$\dim \bigl(\N(E_1,X)\cap\N(E_2,X)\bigr)\geq \rho_X-3,$$
namely by \eqref{eq0}:
$$\dim\bigl(\N(E_0\cap E_1,X)\cap\N(E_0\cap E_2,X)\bigr)\geq\rho_X-3.$$
Again, $\N(E_0\cap E_1,X)$ and $\N(E_0\cap E_2,X)$ have both dimension $\rho_X-3$ (see \ref{typea}), and  using again \eqref{eq0} we conclude that:
 $$\N(E_0\cap E_1,X)=\N(E_0\cap E_2,X)=\N(E_1,X)\cap\N(E_2,X).$$

Let $i\in\{1,2\}$. Since $E_1\cap E_2=\emptyset$, we have $\N(E_{3-i},X)\subseteq E_{i}^{\perp}$, and hence:
$$\N(E_0\cap E_i,X)\subset E_i^{\perp}.$$
After \cite[Rem.~3.1.4]{codim}, this implies that $\R_{\geq 0}[e_i]$ is an extremal ray of type $(n-1,n-2)^{sm}$, and if $\alpha_i\colon X\to Y_i$ is its contraction, then $Y_i$ is Fano.

Since $Y_i$ is smooth and Fano, we can consider a special MMP for $-\alpha_i(E_0')$ in $Y_i$. Together with $\alpha_i$, this yields a special MMP for $-E_0'$ in $X$, where
the first extremal ray is  $\R_{\geq 0}[e_i]$; by \eqref{noa} this MMP is of type $(a)$. Notice also that the extremal ray  $\R_{\geq 0}[e_i]$ cannot be contained in $\N(E_0',X)$, because $\N(E_0',X)\subset E_0^{\perp}$, while $E_0\cdot e_i>0$. Therefore one of the two special indices of this MMP of type $(a)$ is $1$ (see \ref{typea}), 
  and in this way
we get two pairs of disjoint exceptional $\pr^1$-bundles $E_1,\w{E}_2$ (for $i=1$) and 
$\w{E}_1,E_2$ (for $i=2$), all distinct from $E_0'$. 

If $E_0$ intersects $\w{E}_2$, then we get $E_1\cdot e_0>0$ by the first part of the proof, hence a contradiction. Similarly if $E_0$ intersects $\w{E}_1$.

Otherwise, $E_0$ is disjoint from  $E_0',\w{E}_1,\w{E}_2$, and  Rem.~\ref{intersection} yields  $\w{E}_1=\w{E}_2$. Then $\w{E}_1$ is disjoint from $E_0,E_1,E_2$, and we get again a contradiction by Rem.~\ref{intersection}.
This show that $E_1\cdot e_0>0$ and $E_2\cdot e_0>0$, and concludes the proof of Step \ref{step1}.
\end{proof}
\begin{step}\label{delpezzo} We describe the special MMP for $-E_0$.
\end{step}
This is the main part of the proof of Th.~\ref{intermediate}. 
 The reader should bear in mind that after steps \ref{4or5} and \ref{step1}, we have two possible configurations, with either $4$ exceptional $\pr^1$-bundles as in fig.~\ref{fig1} (if $E_0'$ intersects both $E_1$ and $E_2$), or $5$ exceptional $\pr^1$-bundles as in fig.~\ref{fig2} (if $E_0'$ does not intersect both $E_1$ and $E_2$).

Recall that $E_1$ and $E_2$ are obtained from a special MMP for $-E_0$. 
We need to study extremal rays having positive intersection with $E_0$, and we need also  to study extremal rays having positive intersection with $E_1$ or $E_2$. Thanks to the divisors $E_0'$ and $E_0''$, we can apply Lemma \ref{4r} to any such extremal ray.
\begin{parg}\label{typeofray}
Let $R$ be an extremal ray  with $E_i\cdot R>0$ for some $i\in\{0,1,2\}$. We show that 
 $R$ is of type $(n-1,n-2)^{sm}$, and $\Lo(R)\cap (E_1\cup E_2)\neq\emptyset$. 
Moreover for every $j\in\{0,1,2\}$ the contraction of $R$ is finite on $E_j$, unless $[e_j]\in R$.

\medskip

 Indeed, by steps \ref{4or5} and \ref{step1}, we can apply Lemma \ref{4r} to $R$ (here we use the existence of the divisors $E_0'$ and $E_0''$). In particular,
this implies that $R$ is of type $(n-1,n-2)^{sm}$.

If $E_0\cdot R>0$, the properties above follow from  Lemma \ref{4r}. 

Suppose instead that $E_1\cdot R>0$. Then clearly  
  $\Lo(R)\cap (E_1\cup E_2)\neq\emptyset$. By  Lemma \ref{4r} we know that the contraction of $R$ is finite on $E_1$, and also on $E_0$ unless $[e_0]\in R$. 
Finally the contraction of $R$ must be finite on $E_2$ too, because $E_1\cdot R>0$ and $E_1\cap E_2=\emptyset$.

The case $E_2\cdot R>0$ is analogous.
\end{parg}
\begin{parg}\label{posint}
Let 
$R$ be as in \ref{typeofray} and set $E_R:=\Lo(R)$. 
Suppose that  for some $j\in\{1,2\}$ the half-line $\R_{\geq 0}[e_j]$ is an extremal ray of $X$,  different from $R$, and that
  $E_j\cap E_R\neq\emptyset$.
 Then 
 $E_j\cdot R>0$ and  $E_R\cdot e_j>0$.

\medskip

Indeed we have $E_j\cdot R>0$
because the contraction of $R$ is finite on $E_j$ by \ref{typeofray}. 

To show that $E_R\cdot e_j>0$, we proceed by contradiction. Suppose for simplicity that $j=1$.

If $E_R\cdot e_1=0$, then $E_R$ must contain some curve ${e}_1$, and since $E_0\cdot e_1>0$ and $E_0'\cdot e_1>0$, we have $E_0\cap E_R\neq\emptyset$ and $E_0'\cap E_R\neq\emptyset$. 

We have $E_R\neq E_0$ (because $E_0\cdot e_1>0$) and hence
$[e_0]\not\in R$.  By \ref{typeofray}
the contraction of $R$ must be finite on $E_0$, thus $E_0\cdot R>0$. Since $E_0$ and $E_0'$ are disjoint, this also implies that 
 $E_0'\cdot R>0$ by Rem~\ref{e}.

Thus 
 the contraction of the extremal ray  $\R_{\geq 0}[e_1]$ is not finite on $E_R$, $R\neq\R_{\geq 0}[e_1]$, and $E_0,E_0'$ are disjoint prime divisors having positive intersection with $R$. By 
 Lemma \ref{1} we get $E_0\cdot e_1\leq 0$, a contradiction.
\end{parg}
\begin{parg}\label{i1}
We show that  $i_1=1$ and $[e_1]\in R_1$, the first extremal ray of the special MMP for $-E_0$ (see \ref{basic} and \ref{typea}).

Indeed $R_1$ is an extremal ray of $X$ with $E_0\cdot R_1>0$. 
Therefore  $\Lo(R_{1})\cap (E_1\cup E_2)\neq\emptyset$
by \ref{typeofray}.
Since $E_1\cup E_2$ is contained in the open subset where the birational map
$X\dasharrow X_{i_1}$ is an isomorphism (see \ref{typea}), we conclude that
$i_1=1$, $\Lo(R_1)=E_1$, and $[e_1]\in R_1$.
\end{parg}
\begin{parg}
Let $\sigma_1\colon X\to X_2$ be the contraction of $R_1$. We study the next step of the special MMP in $X_2$, which is given by a birational extremal ray $R_2$ of $X_2$ having positive intersection with $\sigma_1(E_0)$. 
We show that  $i_2=2$ and that $R_2$ contains $[\sigma_1(e_2)]$.

We have $R_2=(\sigma_1)_*(\widetilde{R}_2)$, $\widetilde{R}_2$ an extremal ray of $X$ different from $R_1$, such that $R_1+\widetilde{R}_2$ is a face of $\NE(X)$.

Since  $\sigma_1(E_0)\cdot R_2>0$, the projection formula yields
$$\sigma_1^*(\sigma_1(E_0))\cdot \widetilde{R}_2=\bigl(E_0+(E_0\cdot e_1)E_1\bigr)\cdot \widetilde{R}_2>0,$$
hence either $E_0\cdot \widetilde{R}_2>0$ or $E_1\cdot \widetilde{R}_2>0$. 

By \ref{typeofray}, this implies that $\w{R}_2$ is divisorial, of type $(n-1,n-2)^{sm}$, and that $\w{E}_2\cap (E_1\cup E_2)\neq\emptyset$, where $\w{E}_2:=\Lo(\w{R}_2)$. Moreover the contraction of $\w{R}_2$ must be finite on $E_1$, because $R_1\neq \w{R}_2$.

In particular $R_2$ is divisorial and $\Lo(R_2)=\sigma_1(\w{E}_2)$. If $C\subset X$ is an irreducible curve with class in $\w{R}_2$, we have $[\sigma_1(C)]\in R_2$, and:
$$0>\Lo(R_2)\cdot\sigma_1(C)=\bigl(\w{E}_2+(\w{E}_2\cdot e_1)E_1\bigr)\cdot C=
(\w{E}_2\cdot e_1)(E_1\cdot C)-1.$$
This implies that either $E_1\cdot \w{R}_2=0$, or $\w{E}_2\cdot e_1=0$. 
By \ref{posint}, this can happen only if $E_1\cap\w{E}_2=\emptyset$, therefore
we get $E_2\cap \w{E}_2\neq\emptyset$. This yields
$\Lo(R_2)\cap \sigma_1(E_2)\neq\emptyset$.

Since $\sigma_1(E_2)$ is contained in the open subset where the birational map $X_2\dasharrow X_{i_2}$ is an isomorphism (see \ref{typea}), we conclude that
 $i_2=2$, $\w{R}_2=\R_{\geq 0}[e_2]$,  $R_2=\R_{\geq 0}[\sigma_1(e_2)]$, and $\w{E}_2=E_2$.
\end{parg}
\begin{parg}\label{3rdstep}
Let $\sigma_2\colon X_2\to X_3$ be the contraction of $R_2$. Let us consider the next step of the MMP, and the contraction of the associated ray $\ph\colon X_3\to Y$ (so that a priori $\ph$ could be a birational contraction). We set  $\sigma:=\sigma_2\circ\sigma_1\colon X\to X_3$ and  $\psi:=\ph\circ\sigma\colon X\to Y$.
$$
\xymatrix{X\ar@/_1pc/[drr]_{\psi}\ar[r]_{\sigma_1}\ar@/^1pc/[rr]^{\sigma} &{X_2}\ar[r]_{\sigma_2}&{X_3}\ar[d]^{\ph}\\
&&Y
}
$$
We have that $\sigma(E_0)\cdot\NE(\ph)>0$, and $\NE(\psi)$ is a $3$-dimensional face of $\NE(X)$, containing the face $R_1+\w{R}_2$.
\end{parg}
\begin{parg}\label{3rdray}
Let $R$ be an extremal ray of  $\NE(\psi)$, different from $R_1$ and $\w{R}_2$. We have
$\sigma_*(R)=\NE(\ph)$, thus $\sigma(E_0)\cdot\sigma_*(R)>0$, which yields by the projection formula:
$$\bigl(E_0+(E_0\cdot e_1)E_1+(E_0\cdot e_2)E_2\bigr)\cdot R>0.$$
This means that $E_i\cdot R>0$ for some $i\in\{0,1,2\}$, hence \ref{typeofray} and \ref{posint} apply to $R$. In particular, $R$ is of type 
 $(n-1,n-2)^{sm}$, and $\Lo(R)$ is different from $E_1$ and $E_2$.
\end{parg}
\begin{parg}
Let us consider the extremal rays $\w{R}_3$ and $\w{R}_4$ of $\NE(\psi)$, different from $R_1$ and $\w{R}_2$, such that $R_1+\w{R}_3$ and $\w{R}_2+\w{R}_4$ are faces of $\NE(\psi)$ (see fig.~\ref{cone}). Notice that a priori it may be $\w{R}_3=\w{R}_4$, if $\NE(\psi)$ happens to be simplicial. By \ref{3rdray} $\w{R}_3$ and $\w{R}_4$ are of type $(n-1,n-2)^{sm}$; set $E_i:=\Lo(\w{R}_i)$ for $i=3,4$.
\stepcounter{thm}
\begin{figure}[h]
\begin{center} 
\scalebox{0.4}{\includegraphics{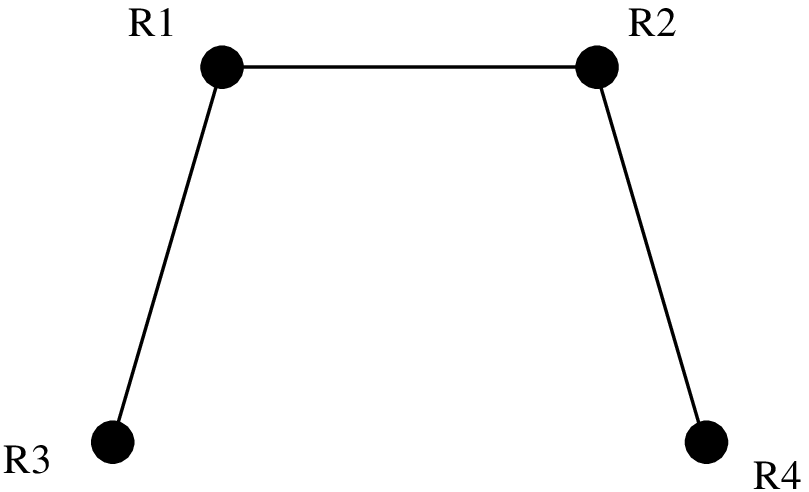}}
\caption{A section of the $3$-dimensional cone $\NE(\psi)$, where dots correspond to extremal rays.}\label{cone}
\end{center}
\end{figure}
\end{parg}
\begin{parg}
We show that $E_1\cap E_3=\emptyset$ and $E_2\cap E_4=\emptyset$.

By contradiction, suppose that 
 $E_1\cap E_3\neq\emptyset$.  Then we have
$E_1\cdot \w{R}_3>0$
and  $E_3\cdot R_1>0$ by \ref{3rdray} and \ref{posint}. 

Since $R_1+\w{R}_3$ is a face of $\NE(X)$, $(\sigma_1)_*(\w{R}_3)$ is
an extremal ray  of $\NE(X_2)$;
let $\zeta\colon X_2\to Z$ be its contraction.
We have
$$\sigma_1(E_3)\cdot \sigma_1(e_3)=\bigl(E_3+(E_3\cdot e_1)E_1\bigr)\cdot e_3=
(E_3\cdot e_1)(E_1\cdot e_3)-1\geq 0,$$
so that $\zeta$ is of fiber type. 

Suppose that $\sigma_1(E_0)\cdot \sigma_1(e_3)>0$.
Then the sequence 
$$X\stackrel{\sigma_1}{\la}X_2\stackrel{\zeta}{\la}Z$$
is a special MMP for $-E_0$ with only one birational map, 
namely $k=2$ in the notation of \ref{basic}, so that the MMP must be of type $(b)$ (see \ref{typea}). This contradicts our assumption
  \eqref{noa}.

If instead $\sigma_1(E_0)\cdot \sigma_1(e_3)\leq 0$, 
by the projection formula we have
$$0\geq \bigl(E_0+(E_0\cdot e_1)E_1\bigr)\cdot e_3=(E_0\cdot e_1)(E_1\cdot e_3)+E_0\cdot e_3.$$
Since $E_0\cdot e_1>0$ and $E_1\cdot e_3>0$, this yields $E_0\cdot e_3<0$ and hence $E_3=E_0$. By \ref{typeofray}, this also 
 implies that $[e_0]\in \w{R}_3$. 

The morphism $\zeta\circ\sigma_1\colon X\to Z$ has a second factorization in elementary contractions:
\stepcounter{thm}
\begin{equation}\label{newMMP}
X\stackrel{\widehat{\sigma}_1}{\la}\widehat{X}_2\stackrel{\widehat{\zeta}}{\la}Z
\end{equation}
where $\widehat{\sigma}_1$ is the divisorial contraction of $\w{R}_3$, and $\widehat{\zeta}$ is the contraction of the extremal ray  $(\widehat{\sigma}_1)_*(R_1)$. Since $\dim Z<n$, $\widehat{\zeta}$ is of fiber type.

Notice that $E_2\cdot \w{R}_3>0$, and
in $\widehat{X}_2$ we have:
$$\widehat{\sigma}_1(E_2)\cdot\widehat{\sigma}_1(e_1)=\bigl(E_2+(E_2\cdot e_0)E_0\bigr)\cdot e_1=
(E_2\cdot e_0)(E_0\cdot e_1)>0.$$
Therefore \eqref{newMMP} is a special MMP for $-E_2$. 
As before, this MMP contains only one birational map and hence is of type $(b)$,
 contradicting again \eqref{noa}.

Finally, if $E_2\cap E_4\neq\emptyset$ we proceed similarly and we get a special MMP of type $(b)$ for either $-E_0$ or $-E_1$, which is again a contradiction.

Therefore $E_1\cap E_3=\emptyset$ and $E_2\cap E_4=\emptyset$.
\end{parg}
\begin{parg}\label{est}
Since $E_3$ is disjoint from $E_1$, we have $E_2\cap E_3\neq\emptyset$ by  \ref{3rdray} and \ref{typeofray}, and similarly $E_1\cap E_4\neq\emptyset$.
 In particular $E_3\neq E_4$ and $\w{R}_3\neq \w{R}_4$. 

Therefore $\sigma(E_3)$ and $\sigma(E_4)$ are distinct prime divisors in $X_3$, both contained in $\Lo(\ph)$, because $\ph$ contracts both $\sigma(e_3)$ and $\sigma(e_4)$. This implies that $\ph$ and $\psi$ are of fiber type.

Moreover by \ref{3rdray} and \ref{posint} we have $E_1\cdot e_4>0$, $E_4\cdot e_1>0$, $E_2\cdot e_3>0$, and $E_3\cdot e_2>0$.

We show that $E_3\cap E_4\neq \emptyset$. This also implies that $E_3\cdot e_4>0$ and $E_4\cdot e_3>0$ by Rem.~\ref{e} (see fig.~\ref{exc}). 

To prove that $E_3$ and $E_4$ are not disjoint, we need to distinguish the cases of $5$ and $4$ exceptional $\pr^1$-bundles, which will be treated separately in paragraphs \ref{5} and \ref{4} respectively. 
\stepcounter{thm}
\begin{figure}[h]
\begin{center} 
\scalebox{0.4}{\includegraphics{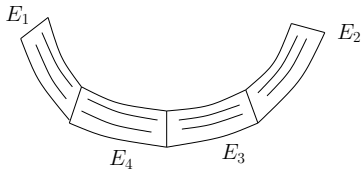}}
\caption{The exceptional divisors $E_1,E_2,E_3,E_4$.}\label{exc}
\end{center}
\end{figure}
\end{parg}
\begin{parg}[Case 
  of $5$ exceptional $\pr^1$-bundles (step \ref{4or5}$(ii)$)]\label{5}
In this case $E_1$ is disjoint from $E_2$, $E_0''$, and $E_3$, and by Rem.~\ref{intersection} we must have $E_3=E_0''$. Similarly, we conclude that $E_4=E_0'$.
Therefore 
 $E_3\cap E_4$ is non-empty. 
\end{parg}
\begin{parg}[Case 
  of $4$ exceptional $\pr^1$-bundles (step \ref{4or5}$(i)$)]\label{4}
 Consider the curves $e_1,e_2,e_3,e_4$, which are all contracted by $\psi$. Since $\dim\ker\psi_*=3$, we have a non-trivial relation of numerical equivalence among these curves. 
Moreover, since $R_1+\w{R}_2$, $R_1+\w{R}_3$, and $\w{R}_2+\w{R}_4$ are faces of $\NE(\psi)$ (see fig.~\ref{cone}), we can assume that the relation is:
$$a_1e_1+a_4e_4\equiv a_2e_2+a_3e_3,$$
with $a_i\in\R_{>0}$ for $i=1,2,3,4$. 

Set $\gamma:=a_1[e_1]+a_4[e_4]\in\N(X)$.
Notice that $E_0'\neq E_4$, because $E_0'$ intersects $E_2$ while $E_2\cap E_4=\emptyset$, thus $E_0'\cdot e_4\geq 0$ and hence 
$$E_0'\cdot\gamma=a_1 E_0'\cdot e_1+a_4E_0'\cdot e_4>0.$$

If by contradiction $E_3\cap E_4=\emptyset$, then $(E_1\cup E_4)\cap (E_2\cup E_3)=\emptyset$. By Rem.~\ref{inter2}, we get:
$$\N(E_1,X)=E_2^{\perp}\cap E_3^{\perp}=\N(E_4,X)\text{ and }\N(E_2,X)=E_1^{\perp}\cap E_4^{\perp}=\N(E_3,X).$$ 

Therefore $\gamma=a_1[e_1]+a_4[e_4]=a_2[e_2]+a_3[e_3]$ belongs to $\N(E_1,X)\cap\N(E_2,X)$.
On the other hand by Lemma \ref{directsum} we know that
$$\N(E_1,X)\cap\N(E_2,X)=\N(E_0\cap E_1)\cap\N(E_0\cap E_2,X)\subset (E_0')^{\perp},$$
because $E_0\cap E_0'=\emptyset$. This is impossible, as $E_0'\cdot\gamma\neq 0$.
\end{parg}
\begin{parg}\label{Ai}
Set $A_i:=\sigma(E_i)\subset X_3$ for $i=1,2$, so that $A_i$ is smooth, irreducible, of dimension $n-2$, $A_1\cap A_2=\emptyset$, and $\sigma\colon X\to X_3$ is the blow-up of $A_1\cup A_2$ (see fig.~\ref{X3}).
We show that $\ph$ is finite on $A_i$ and $\ph(A_i)=Y$, for $i=1,2$; in particular $\dim Y=n-2$.

In $X_3$ we have: 
$$\sigma(E_3)\cdot \sigma(e_4)=\bigl(E_3+(E_3\cdot e_2)E_2\bigr)\cdot e_4= E_3\cdot e_4>0$$
(see \ref{est}).
As $[\sigma(e_4)]\in\NE(\ph)$, we deduce that
$$
\sigma(E_3)\cdot \NE(\ph)>0.$$
Since $\ph$ is of fiber type, this means
 that  
$\sigma(E_3)$ dominates $Y$. On the other hand $\ph$ contracts $\sigma(e_3)$, therefore $\dim Y\leq n-2$.

We have $A_1\cap\sigma(E_3)=\emptyset$, hence  $\sigma(E_3)\cdot C=0$
for every curve  $C\subset  A_1$. This implies that  $[C]\not\in\NE(\ph)$, and $\ph$ is finite on $A_1$. Thus $\ph(A_1)=Y$ and $\dim Y=n-2$.

In a similar way, we see that $\sigma(E_4)\cdot \NE(\ph)>0$,
$\ph$ is finite on $A_2$, and $\ph(A_2)=Y$.
\stepcounter{thm}
\begin{figure}[h]
\begin{center} 
\scalebox{0.4}{\includegraphics{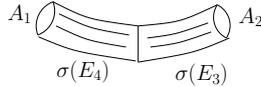}}
\caption{The images of the exceptional divisors in $X_3$.}\label{X3}
\end{center}
\end{figure}
\end{parg}
\begin{parg}
 In particular \ref{Ai} implies that $\ph$ and $\psi$ are equidimensional  fibrations in Del Pezzo surfaces, and that $\psi$ is quasi-elementary. Since $X_3$ is smooth and $\ph$ is $K$-negative, then $Y$ is factorial, has canonical singularities, and $\codim\Sing(Y)\geq 3$, by Lemma \ref{quasiel} below. Thus we have $(ii)$, and
this concludes the proof of Th.~\ref{intermediate}.
\end{parg}
\end{proof}
\begin{lemma}\label{quasiel}
Let $X$ be a smooth projective variety and $f\colon X\to Y$ a $K$-negative, quasi-elementary contraction of fiber type. Then $Y$ is factorial, has canonical singularities, and $\codim\Sing(Y)\geq 3$.
\end{lemma}
\begin{proof}
We know that $Y$ is factorial and  has canonical singularities by \cite[Lemma 3.10(i)]{fanos}. Let us show that $\codim\Sing(Y)\geq 3$; this
 follows with standard techniques from the case where $\dim Y=2$.

Set $m:=\dim Y$. If $m=1$, then $Y$ is smooth and there is nothing to prove. Suppose that $m\geq 2$, and let $H_1,\dotsc,H_{m-2}$ be general elements of a very ample linear system on $Y$. Then $S:=H_1\cap\cdots\cap H_{m-2}$ is an irreducible, normal surface, and $X_S:=\ph^{-1}(S)$ is smooth by Bertini. Moreover $\ph_S:=\ph_{|X_S}\colon X_S\to S$ is a $K$-negative contraction.

The locus of $Y$ where $\ph$ is not equidimensional has codimension at least $3$ \cite[Lemma 3.9(ii)]{fanos}, so that $\ph_S$ is equidimensional. Therefore $S$ is smooth by \cite[Lemma 3.10(ii)]{fanos}. Thus $S\cap\Sing(Y)=\emptyset$, hence $\codim\Sing(Y)\geq 3$.
\end{proof}
\section{The conic bundle case}\label{cb}
In this section we 
first describe, in Prop.~\ref{conicbundle}, the case where there exists a special MMP of type $(b)$ for $-D$, $D$ a prime divisor in $X$ with $c(D)=2$.
Then we use Th.~\ref{intermediate} and Prop.~\ref{conicbundle} to prove 
 Th.~\ref{main2} (which is a more detailed version of our main result, Th.~\ref{main}).
\begin{proposition}\label{conicbundle}
Let $X$ be a smooth Fano variety with $c_X=2$, $D\subset X$ a prime divisor with $c(D)=2$, and suppose that there is a special MMP of type $(b)$ for $-D$. We keep the same notation as in \ref{basic} and \ref{typeb}. 
\stepcounter{thm}
\begin{equation}\label{dMMP}
\xymatrix{
{X=X_1}\ar@/^1pc/@{-->}[rrrr]^{\sigma}\ar@/_1pc/@{-->}[drrrr]_{\psi}
\ar@{-->}[r]_{\,\quad \sigma_1}& 
{X_2}\ar@{-->}[r]&
{\cdots}\ar@{-->}[r] &
{X_{k-1}}\ar@{-->}[r]_{\sigma_{k-1}}& 
{X_k}\ar[d]^{\ph}\\
&&&&Y 
}\end{equation}
Then  $X_k$ and $Y$ are smooth, $\ph\colon X_k\to Y$ is a conic bundle,
$\sigma_{j}$ is a flip for every $j\neq i_1$, and $\rho_X-\rho_Y=2$.

Moreover, we have one of
the following:
\begin{enumerate}[$(i)$]
\item  $\ph$ is  smooth;
\item  $k=2$, $\sigma=\sigma_1\colon X\to X_2$ is the blow-up of a smooth subvariety of codimension $2$, $\psi\colon X\to Y$ is a conic bundle, and there exists a smooth $\pr^1$-fibration $\zeta\colon Y\to Y'$, $Y'$ a smooth projective variety of dimension $n-2$.
$$\xymatrix{
 X\ar[d]_{\psi'}\ar[dr]^{\psi}\ar[r]^{\sigma} & {X_{2}}\ar[d]^{\ph} \\
{Y'} & Y\ar[l]^{\zeta}
}$$
 Moreover  $\psi':=\zeta\circ\psi\colon X\to Y'$ is an  equidimensional,
 quasi-elementary fibration in Del Pezzo surfaces, and $\rho_X-\rho_{Y'}=3$.
\end{enumerate}
\end{proposition}
\begin{proof}
Let $\ell\subset X$ be a general fiber of the rational conic bundle $\psi\colon X\dasharrow Y$; recall that $\ell\cap (E\cup\wi{E})=\emptyset$ (see \ref{typeb}).
Lemma \ref{min} shows that $X_k$ and $Y$ are smooth, and $\ph$ is a conic bundle.
 \begin{parg}\label{a}
We show that $\sigma_{j}$ is a flip for every $j\neq i_1$.
 By contradiction, let  $\sigma_{j}$ be the first divisorial contraction (different from $\sigma_{i_1}$) occurring in the special MMP. 

Let $G\subset X$  be the transform of $\Exc(\sigma_j)$. 
We have $G\cap(E\cup\wi{E})=\emptyset$ (see \ref{typeb}), so that $\N(G,X)= E^{\perp}\cap\wi{E}^{\perp}$ by Rem.~\ref{inter2}. 
Since $E\cdot\ell=\wi{E}\cdot\ell=0$, this implies that $[\ell]\in\N(G,X)$.
Using Lemma \ref{flips}, we deduce that
  $[\ell]\in\N(D_k,X_k)$, a contradiction because $[\ell]\in\ker\ph_*$ and $\ker\ph_*\not\subset\N(D_k,X_k)$ (see \ref{typeb}).
\end{parg}
\begin{parg}
  To show the last part of the statement, 
\emph{we assume for the rest of the proof 
 that the conic bundle $\ph\colon X_k\to Y$ is not smooth,} and we prove $(ii)$. 
Let us sketch our strategy, which is similar to the one used in \cite[\S 3.3]{codim} to study the case $c_X=3$.

The assumption that $\ph$ is not smooth yields an effective divisor $\Delta\subset X$ disjoint from $E\cup \wi{E}$ (given by singular fibers of $\ph$, see \ref{c}). This implies that $c(E)=2$ (paragraph \ref{cE}), and by looking at a special MMP for $-E$, we find an exceptional $\pr^1$-bundle $F\subset X$ which dominates $Y$, as in fig.~\ref{delta} (paragraphs \ref{F} and \ref{positive}). We use $F$ to prove that the MMP \eqref{dMMP} has no flips, in paragraphs \ref{flipdisj}--\ref{sabrina}, so that $\psi\colon X\to Y$ is regular and is a conic bundle. Finally, we show that the $\pr^1$-bundle on $F$ induces a smooth $\pr^1$-fibration on $Y$ (paragraph \ref{unsplit}).
\end{parg}
\begin{parg}\label{c}
Since $\ph$ is a conic bundle and is not smooth, the union
$\Delta\subset X_k$ of its singular fibers is a non-empty reduced effective divisor in $X_k$. 
Moreover, every fiber of $\ph$ over $\ph(T)$ is integral (recall that $T\subset X_k$ is the indeterminacy locus of $\sigma^{-1}\colon X_k\dasharrow X$, see \ref{typeb}), hence $\ph$ is smooth over $\ph(T)$, namely: $$\Delta\cap\ph^{-1}(\ph(T))=\emptyset.$$

This shows that $\Delta$ is contained in the open subset where
$\sigma^{-1}\colon X_k\dasharrow X$ is an isomorphism, and that
 in $X$ we have:
\stepcounter{thm}
\begin{equation}\label{cabofrio}
\Delta\cap (E\cup\widehat{E})=\emptyset,\end{equation}
where we still denote by $\Delta$ the transform of $\Delta$ in $X$. As $\Delta$ may be reducible, we fix an irreducible component $\Delta_0\subset X$; notice that $\Delta_0$ is the inverse image under $\psi$ of an irreducible component of the discriminant locus of $\ph$ in $Y$. 
\end{parg}
\begin{parg}\label{sim}
We show that
$\N(E,X)\cap\N(\Delta_0,X)\subseteq\N(D,X)$.

Let $\gamma\in \N(E,X)\cap\N(\Delta_0,X)$. Since $\gamma\in\N(E,X)$ and $D\cdot e>0$, by Rem.~\ref{elem2}(1) we have
\stepcounter{thm}
\begin{equation}\label{first}
\gamma=\lambda [e]+\eta,\end{equation}
with $\lambda\in\R$ and $\eta\in\N(D\cap E, X)\subseteq\N(D,X)$.

On the other hand, since $D\cdot\ell>0$, we also have 
$$\N(\Delta_0,X)=\R[\ell]+\N(D\cap\Delta_0,X)$$
(see \cite[Prop.~II.4.19]{kollar}).
Thus $\gamma\in\N(\Delta_0,X)$ yields
\stepcounter{thm}
\begin{equation}\label{second}
\gamma=\lambda'[\ell]+\eta'=\lambda' [e]+\lambda' [\wi{e}\,]  +\eta',
\end{equation}
where $\lambda'\in\R$ and $\eta'\in\N(D\cap \Delta_0,X)\subseteq\N(D,X)$ (recall that $\ell\equiv e+\wi{e}$, see \ref{typeb}).

By Lemma \ref{directsum0}, comparing \eqref{first} and \eqref{second}, we deduce that $\lambda=\lambda'=0$ and $\gamma=\eta\in\N(D,X)$.
\end{parg}
\begin{parg}\label{cE}
We show that $c(E)=2$.

Indeed by \eqref{cabofrio} we have 
$$\N(E,X)\cup\N(\wi{E},X)\subseteq\Delta^{\perp}.
$$
The subspaces $\N(E,X)$ and $\N(\wi{E},X)$ have dimension at least $\rho_X-2$, because $c_X=2$. On the other hand  $\N(E,X)$
 cannot contain $\N(\wi{E},X)$, because
$[\wi{e}\,]\not\in\N(E,X)$ (see \ref{typeb}). This gives the statement.
\end{parg}
\begin{parg}\label{F}
There exists an exceptional $\pr^1$-bundle $F\subset X$ such that $E\cdot f>0$, $F\neq E$, and $F\neq\wi{E}$. In particular, $F\cap E\neq\emptyset$.

Indeed, as $c(E)=2$,   we can consider a special MMP for $-E$. By \ref{typea} and \ref{typeb}, either we get $F$ as above, or 
the only possibility is that the special MMP is of type $(b)$, 
and there is an exceptional $\pr^1$-bundle 
 $\w{E}\subset X$ such that $\w{E}\cdot e=0$ and $\w{E}\cdot\wi{e}=1$. Then we get $\w{E}\cdot \ell=\w{E}\cdot(e+\wi{e}\,)=1$, a contradiction because in this case every fiber of $\ph$ should be smooth.
\end{parg}
\begin{parg}\label{positive}
We show that: $\Delta_0\cdot f>0$, $F\cdot e>0$, and $F\cdot \ell>0$.
\stepcounter{thm}
\begin{figure}[h]
\begin{center} 
\scalebox{0.3}{\includegraphics{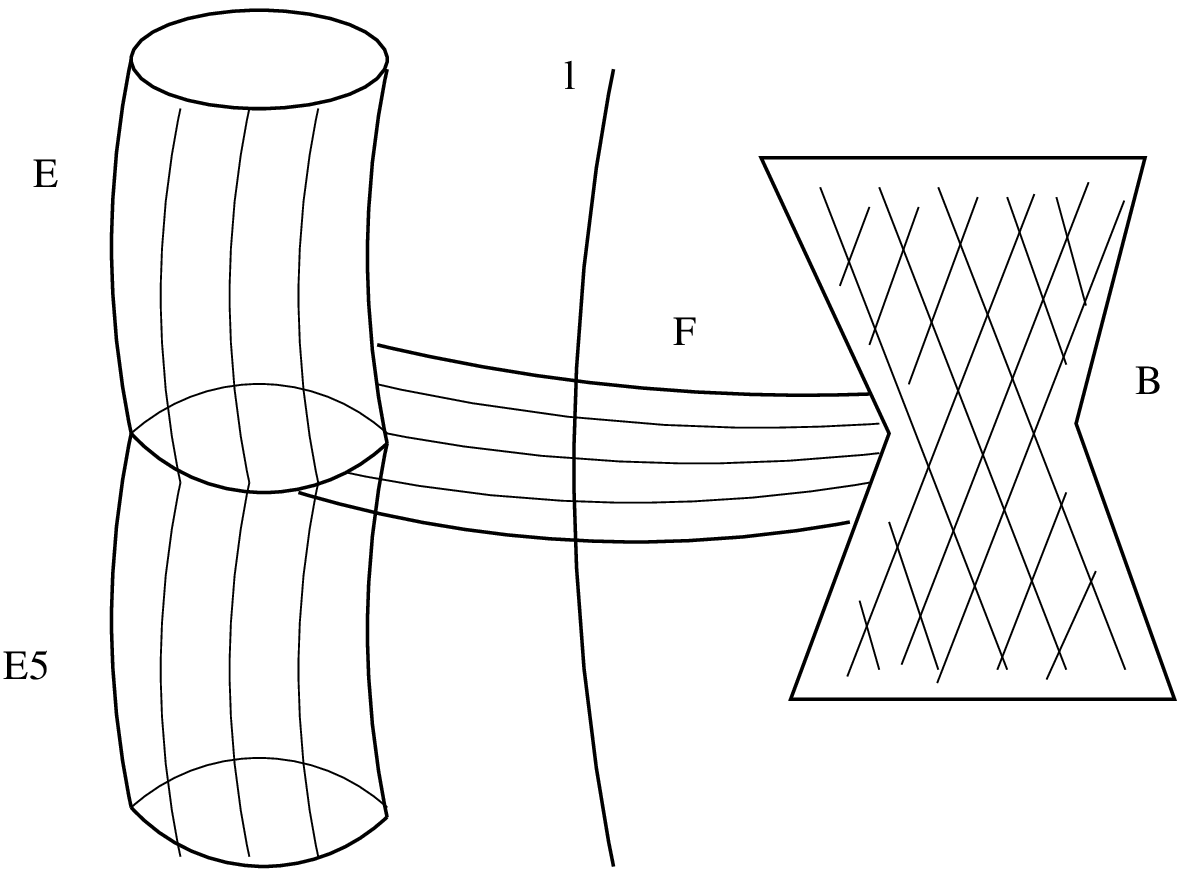}}
\caption{}
\label{delta}
\end{center}
\end{figure}

Indeed $\Delta_0$ is disjoint from $E$ and $\wi{E}$, so that $\Delta_0\cap F\neq\emptyset$ by Rem.~\ref{intersection}. Then $\Delta_0\cdot f>0$ by Rem.~\ref{e}, as
 $E\cdot f>0$ and  $E\cap\Delta_0=\emptyset$. 

We have $F\cdot e\geq 0$ because $F\neq E$. By contradiction, if $F\cdot e=0$, then there exists some curve ${e}$ contained in $F$, and since $\Delta_0\cdot f>0$, by Rem.~\ref{elem2}(3)  we have
$${e}\equiv \lambda f+\mu C$$
where $C$ is a curve contained in $\Delta_0\cap F$ and $\lambda,\mu\in\R$. Intersecting with $E$ we get
$-1=E\cdot {e}=\lambda E\cdot f$, so that $\lambda<0$ (recall that $E\cdot f>0$ by \ref{F}). 
Since $F\neq\wi{E}$, we have $\wi{E}\cdot f\geq 0$, and
intersecting with $\wi{E}$ we get $1=\wi{E}\cdot {e}=\lambda \wi{E}\cdot f\leq 0$ (see \ref{typeb}), a contradiction. Thus $F\cdot e>0$.

Finally $F\cdot\wi{e}\geq 0$ because $F\neq\wi{E}$, hence $F\cdot\ell=F\cdot (e+\wi{e}\,)>0$.
\end{parg}
\begin{parg}\label{Fagain}
If $\Gamma\subset F$ is an irreducible curve such that $\Gamma\cap(E\cup\Delta_0)=\emptyset$, then $\Gamma\equiv\mu_1\Gamma_1\equiv\mu_2\Gamma_2$, where $\mu_1,\mu_2\in\R$, $\Gamma_1$ is a curve contained in $E\cap F$, and $\Gamma_2$ is a curve contained in $\Delta_0\cap F$.

Indeed as $E\cdot f>0$, by Rem.~\ref{elem2}(3) we have
$$\Gamma\equiv\lambda f+\mu_1 \Gamma_1,$$
where $\lambda,\mu_1\in\R$ and $\Gamma_1$ is a curve contained in $E\cap F$. Intersecting with $\Delta_0$ we get $\lambda\Delta_0\cdot f=0$ and hence
$\lambda=0$, because $\Delta_0\cdot f>0$ by \ref{positive}.

Exchanging the roles of $E$ and $\Delta_0$, we get also the second part of the statement above.
\end{parg}
\begin{parg}\label{flipdisj}
Recall that by \ref{a}, $\sigma_j$ is a flip for every $j\neq i_1$.

We show that the locus of every flip of the MMP is disjoint from $F$. The proof will take paragraphs \ref{flipdisj}, \ref{notfinite}, and \ref{finite}.

By contradiction, let $j_0\in\{1,\dotsc,k-1\}\smallsetminus\{i_1\}$ be the first index such that $\Lo(R_{j_0})\cap F_{j_0}\neq\emptyset$, where $F_{j_0}$ is the transform of $F$ in $X_{j_0}$.
Set
$$\gamma:=\sigma_{j_0-1}\circ\cdots\circ\sigma_1\colon X\dasharrow X_{j_0},$$
and notice that $\gamma$ is regular around both $E$ (see \ref{typeb}) and $F$ (by our choice of $j_0$). Moreover, $\gamma(E)\cap\Lo(R_{j_0})=\emptyset$ (again \ref{typeb}).
 \end{parg}
\begin{parg}\label{notfinite}
Suppose  that the contraction of  $R_{j_0}$  is not finite on $F_{j_0}$, let $C\subset F_{j_0}$ be an irreducible curve with class in 
$R_{j_0}$,
 and let $\w{C}\subset F$ be its transform in $X$. Then $\w{C}$ is disjoint from $E$ and $\Delta_0$, because these divisors are disjoint from the loci of all flips of the MMP.

By \ref{Fagain}, we have 
 $\w{C}\equiv\mu C'$, where $\mu\in\R$ and $C'$ is a curve contained in $E\cap F$.
Both $\w{C}$ and $C'$ are contained in $F$, on which the birational map $\gamma\colon X\dasharrow X_{j_0}$ is regular; set $C'':=\gamma(C')\subset X_{j_0}$. 

Since $C=\gamma(\w{C})$, in $X_{j_0}$ we get $C\equiv \mu' C''$ with $\mu'\in\R$, hence $[C'']\in R_{j_0}$. But this is impossible, because $C''\subset\gamma(E)$ and $\gamma(E)\cap\Lo(R_{j_0})=\emptyset$.
\end{parg}
\begin{parg}\label{finite}
Suppose now that the contraction of  $R_{j_0}$ is finite on $F_{j_0}$. Since $\Lo(R_{j_0})\cap F_{j_0}\neq\emptyset$,
we have  $F_{j_0}\cdot R_{j_0}>0$, and every non-trivial fiber of the contraction of $R_{j_0}$ must have dimension one. Then Prop.~\ref{small}$(iv)$ implies that $\Lo(R_{j_0})$ is contained in the indeterminacy locus $W$
of $\gamma^{-1}$.

If $i_1>j_0$, then $\gamma$ is a composition of flips, and by our choice of $j_0$ we have $F_{j_0}\cap W=\emptyset$, a contradiction.

If instead $i_1<j_0$, then $W=\gamma(E)\cup W'$, where $W'$ is disjoint from $\gamma(E)$ and is given by the union of the loci of the flips preceding $\sigma_{j_0}$. We have $\Lo(R_{j_0})\subseteq W'$, because $\Lo(R_{j_0})\cap \gamma(E)=\emptyset$.
As before this gives a contradiction, because  by our choice of $j_0$ we have $F_{j_0}\cap W'=\emptyset$.

Therefore every flip of the MMP is disjoint from $F$.
\end{parg}
\begin{parg}\label{sabrina}
We show that $k=2$, so that $i_1=1$ and the MMP \eqref{dMMP} has just one step. By \ref{a}, we have to exclude the presence of flips.

By contradiction, let $\sigma_m\colon X_m\dasharrow X_{m+1}$ be the last flip 
of the MMP, so that either $m=k-1$ and $i_1<k-1$, or $m=k-2$ and $i_1=k-1$.
Notice that we have an induced conic bundle $\ph'\colon X_{m+1}\to Y$, where $\ph'=\ph$ if $m=k-1$, and $\ph'=\ph\circ\sigma_{k-1}$ if $m=k-2$ and $i_1=k-1$. Notice also that in the latter case, the indeterminacy locus of $\sigma_{k-2}^{-1}$ in $X_{k-1}$ is disjoint from $(\ph')^{-1}(Z)$ (recall that $Z=\psi(E)=\psi(\wi{E})\subset Y$).

For simplicity we assume $m=k-1$, the other case being similar.
 
Let $C\subset X_{k}$ be an irreducible curve with class in the small ray $R_{k-1}'$ corresponding to the flip $\sigma_{k-1}\colon X_{k-1}\dasharrow X_k$; notice that $C\subset D_k$, because $D_k\cdot R_{k-1}'<0$.

Set $S:=\ph^{-1}(\ph(C))$. Since $C\subset T$ and $\ph$ is smooth over $\ph(T)$ (see \ref{c}), we have 
\stepcounter{thm}
\begin{equation}
\label{ultima0}
S\cap\left(\ph^{-1}(Z)\cup \Delta\right)=\emptyset.\end{equation}

Set $F_k:=\sigma(F)\subset X_k$.  
The surface $S$ intersects $F_{k}$, because $F_{k}\cdot\ell>0$. On the other hand $F_{k}\cap C=\emptyset$ by \ref{flipdisj}, so that $\dim(S\cap F_{k})=1$. 

Let $C_2\subset S\cap F_{k}$ be an irreducible component of $S\cap F_{k}$.
Since $S$ is ruled, we have 
\stepcounter{thm}
\begin{equation}
\label{ultima}
C_2\equiv a \ell+b C\qquad\text{ with $a,b\in\R$.}
\end{equation}

Let us consider the transform $\w{C}_2$ of $C_2$ in $X$; notice that by \ref{flipdisj}, $\w{C}_2$ is contained in the open subset where $\sigma\colon X\dasharrow X_k$ is an isomorphism. We have
$$\w{C}_2\subset F\quad\text{and}\quad
 \w{C}_2\cap(E\cup \Delta_0)=\emptyset$$ 
(see \eqref{ultima0}). Thus \ref{Fagain} implies that:
  $$[\w{C}_2]\in\N(E,X)\cap\N(\Delta_0,X).$$

By \ref{sim}, we get $[\w{C}_2]\in\N(D,X)$, and finally 
Lemma \ref{flips} yields
 that $[C_2]\in\N(D_{k},X_{k})$.

Consider now \eqref{ultima}. We have $[C],[C_2]\in \N(D_k,X_k)$, therefore
 $a[\ell]\in\N(D_{k},X_{k})$. Since $[\ell]\not\in\N(D_{k},X_{k})$ (see \ref{typeb}), we deduce that $a=0$ and $[C_2]\in R_{k-1}'$. This contradicts \ref{flipdisj}, because $C_2\subset F_k$.

We conclude that there are no flips in the MMP, and $k=2$.
\end{parg}
\begin{parg}\label{unsplit}
Since $k=2$, the map $\psi\colon X\to Y$ is regular and is a conic bundle. 
$$\xymatrix{X\ar[r]^{\sigma=\sigma_1}\ar[dr]_{\psi}&{X_2}\ar[d]^{\ph}\\
&Y
}$$
The divisor $F$ dominates $Y$ because $F\cdot\ell>0$, hence $\dim\psi(F)=n-1=\dim F$. In particular, $\psi(f)\subset Y$ is a curve.

Let us consider the prime divisor $Z=\psi(E)\subset Y$; recall that $\psi^*(Z)=E+\wi{E}$.
Consider also $\psi_*\colon\N(X)\to\N(Y)$. We have $\ker\psi_*=\R[e]\oplus\R[\wi{e}\,]$ and $[\wi{e}\,]\not\in\N(E,X)$ (see \ref{typeb}), so $\dim(\ker\psi_*\cap\N(E,X))=1$, $\dim\N(E,X)=\rho_X-2$ (see \ref{cE}), and 
$$\dim\N(Z,Y)=\dim\N(E,X)-1=\rho_X-3=\rho_Y-1.$$
Finally we have:
$$\psi^*(Z)\cdot f=(E+\wi{E})\cdot f>0$$
(see \ref{F}).
Therefore by \cite[Lemma 3.2.25]{codim} there   exist a smooth projective variety $Y'$, and a smooth $\pr^1$-fibration $\zeta\colon Y\to Y'$, 
whose fibers are the curves $\psi(f)\subset Y$.

Hence $\psi'=\zeta\circ\psi\colon X\to Y'$ is an equidimensional fibration in Del Pezzo surfaces, and $\rho_{Y'}=\rho_X-3$. 

Moreover $Z\subset Y$ dominates $Y'$, and $Z=\ph(A)$, where $A\subset X_2$ is the center of the blow-up $\sigma$.
Thus $A$ dominates $Y'$, and this implies that 
$\psi'$ is quasi-elementary.

This concludes the proof of Prop.~\ref{conicbundle}.
\end{parg}
\end{proof}
\begin{thm}\label{main2}
Let $X$ be a smooth Fano variety with
 $c_X=2$. Then one of the following holds:
\begin{enumerate}[$(i)$]
\item  there exists a diagram
$$\xymatrix{
X\ar@{-->}[r]^{\phi}&{X'}\ar@/^1pc/[rr]^f\ar[r]_{\alpha}& {X''}\ar[r]_{\ph} &Y
}$$
where all varieties are smooth and projective, $\phi$ is a sequence of flips, 
$\ph$  is a smooth $\pr^1$-fibration, $\alpha$ is the blow-up of a smooth, irreducible subvariety $A\subset X''$ of codimension $2$, $f$ is a conic bundle, and $\rho_X-\rho_Y=2$.
Moreover $f^{-1}(\ph(A))\subset X'$ is contained in the open subset where $\phi^{-1}$ is an isomorphism;
\item there is an equidimensional, quasi-elementary fibration in Del Pezzo surfaces $\psi\colon X\to Y$, where $Y$ is factorial, has  canonical singularities, $\codim\Sing(Y)\geq 3$, and $\rho_X-\rho_Y=3$.
\end{enumerate}
\end{thm}
Notice that Th.~\ref{main} follows from Th.~\ref{main2}.
\begin{proof}
By Th.~\ref{intermediate}, either we have $(ii)$, or there exist a prime divisor $D\subset X$ with $c(D)=2$ and a special MMP of type $(b)$ for $-D$.

In this last case, we apply Prop.~\ref{conicbundle}. We keep the same notation as in the Proposition. We have that $X_k$ and $Y$ are smooth, $\ph\colon X_k\to Y$ is a conic bundle, and $\rho_X-\rho_Y=2$.

If $\ph$ is not smooth, then we are in Prop.~\ref{conicbundle}$(ii)$, which gives  again $(ii)$.

If $\ph$ is smooth, we set $X'':=X_k$.
The birational map $\sigma\colon X\dasharrow X''$ is a composition of flips and a unique divisorial contraction. 
We can factor it as
$$X\stackrel{\phi}\dasharrow X'\stackrel{\alpha}{\longrightarrow} X'',$$
where $\phi\colon X\dasharrow X'$ is a sequence of flips, and $\alpha\colon X'\to X''$ is regular (see \cite[Prop.~1.11]{hukeel}). Then $\alpha$ must be elementary and divisorial, with exceptional divisor the transform of $E$. Thus $\alpha$ is just the blow-up of $A=\sigma(E)\subset X''$, $X'$ is smooth, and no flip in $\phi$ intersects $E\cup\wi{E}$. 
This gives $(i)$.
\end{proof}
\begin{remark}\label{X0}
Let $X$ be a smooth Fano variety with $c_X=2$, and suppose that $X$ satisfies  Th.~\ref{main2}$(i)$. Then
  there exist an open subset $X_0\subseteq X$ such that $\codim(X\smallsetminus X_0)\geq 2$, and a conic bundle $f_0\colon X_0\to Y_0$ where $Y_0$ is smooth and quasi-projective, such that $f_0$ has relative Picard number two, and factors through a 
 smooth $\pr^1$-fibration over $Y_0$.

Indeed, let $L\subset X'$ be the indeterminacy locus of $\phi^{-1}$. By Lemma \ref{indet}, we have $\codim L\geq 3$, thus $\codim f^{-1}(f(L))\geq 2$. We set $X_0:=\phi^{-1}(X'\smallsetminus f^{-1}(f(L)))$, $Y_0:=Y\smallsetminus f(L)$,  and $f_0:=f\circ\phi_{|X_0}$, and we have the properties above.
\end{remark}
\section{Related results and examples}\label{last}
In this last section we consider the case $c_X=1$, and some special issues of Th.~\ref{main}.
\begin{proposition}\label{codimone}
Let $X$ be a smooth Fano variety with $c_X=1$. Then one of the following holds:
\begin{enumerate}[$(i)$]
\item there are 
an exceptional $\pr^1$-bundle $E\subset X$ and a sequence of flips $\phi\colon X\dasharrow X'$, such that $E$ is contained in the open subset where $\phi$ is an isomorphism, and the transform $E'\subset X'$  is the locus of an extremal ray of type $(n-1,n-2)^{sm}$;
\item there are a contracting birational map $\sigma\colon X\dasharrow X'$, and a conic bundle $\ph\colon X'\to Y$,
 such that $X'$ and $Y$ are smooth and projective, and $\rho_Y=\rho_{X'}-1$.
\end{enumerate}

If moreover $X$ contains two disjoint prime divisors $D$ and $D'$, then  we can replace $(ii)$ by:
\begin{enumerate}
\item[$(ii)'$]
there are a sequence of flips $\sigma\colon X\dasharrow X'$, and a conic bundle $\ph\colon X'\to Y$, such that $D'$ is contained in the open subset where $\sigma$ is an isomorphism,
$X'$ and $Y$ are smooth and projective, $\rho_Y=\rho_{X}-1$, and 
$\ph$ is finite on $\sigma(D')$.
\end{enumerate}\end{proposition}
\begin{proof}
Let $D\subset X$ be a prime divisor with $c(D)=1$, and consider a special MMP for $-D$ as described in \ref{basic}.

Suppose that the MMP is of type $(a)$. We keep the same notation as in \ref{basic}.

By \cite[Lemma 2.7]{codim}, 
there is a special index $i_1\in\{1,\dotsc,k-1\}$ such that:
 $R_i\subset\N(D_i,X_i)$ for every $i\in\{1,\dotsc,k\}\smallsetminus\{i_1\}$,
 $R_{i_1}\not\subset\N(D_{i_1},X_{i_1})$, and $R_{i_1}$ is of type $(n-1,n-2)^{sm}$.
 Moreover, $\Lo(R_{i_1})$ is contained in the open subset where the birational map $X_{i_1}\dasharrow X$ is an isomorphism, and its transform $E\subset X$ is an exceptional $\pr^1$-bundle.

We set $X':=X_{i_1}$ and $\phi:=\sigma_{i_1-1}\circ\cdots\circ \sigma_1\colon X \dasharrow X'$. In order to get $(i)$, we just have
to show that $\sigma_j$ is a flip for every $j<i_1$.

By contradiction, let $j\in\{1,\dotsc,i_1-1\}$ be the first index such that $\sigma_j$ is a divisorial contraction. 
Since $D_j\cdot R_j>0$, we have $\sigma_j(\Exc(\sigma_j))=\sigma_j(D_j\cap\Exc(\sigma_j))$ and hence 
$(\sigma_j)_*(\N(\Exc(\sigma_j),X_j))=(\sigma_j)_*(\N(D_j\cap\Exc(\sigma_j),X_j))$.
On the other hand $\ker(\sigma_j)_*=\R R_j\subseteq \N(D_j,X_j)$, thus
$$\N(\Exc(\sigma_j),X_j)=\R R_j+\N(D_j\cap\Exc(\sigma_j),X_j)\subseteq\N(D_j,X_j).$$

 We show that:
$$\N(G_i,X_i)\subseteq\N(D_i,X_i)\quad\text{ for every $i=1,\dotsc,j$,}$$ where 
$G_i\subset X_i$ is the transform of $\Exc(\sigma_j)\subset X_j$.
Indeed suppose that $\N(G_i,X_i)\subseteq\N(D_i,X_i)$ for some 
 $i\in\{2,\dotsc,j\}$, and consider 
 the diagram associated with the flip:
$$\xymatrix{{X_{i-1}}\ar@{-->}[rr]^{\sigma_{i-1}}\ar[dr]_{\zeta}&&{X_i}\ar[dl]^{\zeta'}\\
&Z&
}$$
We have 
$$\zeta_*(\N(G_{i-1},X_{i-1}))\subseteq \zeta_* (\N(D_{i-1},X_{i-1})).$$
Moreover $\ker\zeta_*=\R R_{i-1}\subseteq\N(D_{i-1},X_{i-1})$, which yields that  $\N(G_{i-1},X_{i-1})\subseteq\N(D_{i-1},X_{i-1})$.

In the end we get
$$\N(G_1,X)\subseteq\N(D,X)\subsetneq\N(X),$$
 and since $c(X)=1$, we conclude that $\N(G_1,X)=\N(D,X)$. However this is impossible, because
$G_1\cap E=\emptyset$ (because $E$ is contained in the locus where $\phi$ is an isomorphism), thus $\N(G_1,X)\subseteq E^{\perp}$. On the other hand $D\cap E\neq\emptyset$ and $D\neq E$, hence there are curves $C\subset D$ with $E\cdot C>0$, namely  $\N(D,X)\not\subseteq E^{\perp}$.

\medskip

Suppose now that the MMP is of type $(b)$. By
\cite[Lemmas 2.7 and 2.8]{codim} we have 
 $R_i\subset\N(D_i,X_i)$ for every $i\in\{1,\dotsc,k-1\}$,  $\dim Y=n-1$, and every fiber of 
$\ph$ has dimension one.
As in Lemma \ref{min}, we see that 
$X_k$ and $Y$ are smooth, and that $\ph\colon X_k\to Y$ is a conic bundle. Thus we set $X':=X_k$ and we have $(ii)$.

\medskip

Finally, suppose that there is a prime divisor $D'$ disjoint from $D$, and let
$D'_i\subset X_i$ be the transform of $D'$, for every $i=1,\dotsc,k$.

We show that: 
\stepcounter{thm}
\begin{equation}\label{roma}
D_i'\cap \bigl(D_i\cup \Lo(R_i)\bigr)=\emptyset\quad\text{ for all }i=1,\dotsc,k-1.
\end{equation}
Indeed suppose that $D_i\cap D'_i=\emptyset$ (which holds for $i=1$). Then we have $\N(D_i,X_i)\subseteq (D_i')^{\perp}$, and conversely  $\N(D'_i,X_i)\subseteq D_i^{\perp}$. 

Since $D_i\cdot R_i>0$ and $R_i\subset \N(D_i,X_i)$, we deduce that  $R_i\not\subset\N(D_i',X_i)$  and that  $D_i'\cdot R_i=0$, which yields
$$\Lo(R_i)\cap D_i'=\emptyset\quad\text{and}\quad D_{i+1}\cap D_{i+1}'=\emptyset.$$
This gives \eqref{roma}. Similarly, we deduce that $D_k\cap D_k'=\emptyset$ and that $R_k\not\subset\N(D_k',X_k)$.

In particular, $D'$ is contained in the open subset where $\sigma\colon X\dasharrow X_k$ is an isomorphism, and $\ph$ is finite on $\sigma(D')\subset X_k$.

Moreover, if $\sigma_j$ is divisorial for some $j\in\{1,\dotsc,k-1\}$, let $B\subset X$ be the transform of $\Exc(\sigma_j)$. Then $D'\cap (D\cup B)=\emptyset$, thus 
$$\N(D',X)\subseteq D^{\perp}\cap B^{\perp}.$$
Since $c_X=1$, we have $\dim\N(D',X)\geq\rho_X-1$, which implies that $D^{\perp}= B^{\perp}$, namely the divisors $B$ and $D$ are numerically proportional. However this is impossible, because since $B$ is a fixed divisor, it is the unique prime divisor having class in the half-line $\R_{\geq 0}[B]\subset\Nu(X)$ (see \ref{notation}). Therefore $\sigma_i$ is a flip for every $i=1,\dotsc,k-1$, and we have $(ii)'$.
\end{proof}
\begin{parg}[The case $c_X=2$ and $\rho_X=3$]
If $X$ is a Fano variety with $c_X=2$, we have $\rho_X\geq 3$. Fano manifolds with $c_X=2$ and $\rho_X=3$ are described in \cite[Th.~3.5]{minimal}; in
particular $X$ always satisfies Th.~\ref{main2}$(i)$ with $X=X'$ (without flips), and $X''$ is the projectivization of a decomposable rank $2$ vector bundle over $Y$.
\end{parg}
\begin{parg}[Small dimensions]
If $X$ is a Del Pezzo surface, then $c_X=\rho_X-1$.

Let $X$ be a Fano $3$-fold. If $\rho_X\geq 4$, then $c_X=\rho_X-2$ by \cite[Lemma 5.1]{gloria} and \cite[Lemma 3.1 and references therein]{minimal}.

Therefore Fano $3$-folds with $c_X=2$ have $\rho_X\in\{3,4\}$;  it follows from Mori and Mukai's classification that they are all
conic bundles over $\pr^2$, $\pr^1\times\pr^1$, or $\mathbb{F}_1$ (see for instance \cite[Th.~3.5]{minimal} for the case $\rho_X=3$, \cite[Th.~7.1.15 and Theorem on p.~141]{fanoEMS} and references therein for the case $\rho_X=4$).

In dimension $4$, a (slightly weaker) version of Th.~\ref{main} has been proved in \cite{eff,cdue}. Using Th.~\ref{main}, one can show that Fano $4$-folds with $c_X=2$ have
$\rho_X\leq 12$, see \cite[Th.~1.2]{cdue}.
\end{parg}
\begin{parg}[Toric Fano manifolds]
A toric version of Th.~\ref{codim}, Th.~\ref{main}, and Prop.~\ref{codimone} 
 was already present in \cite[Th.~3.4]{fano}. Even if the techniques used in the toric case are completely combinatorial,
they rely on a ``basic construction'' \cite[p.~1478]{fano} which can be easily translated geometrically in a special MMP for a torus-invariant prime divisor $D\subset X$. This special MMP is very explicit, and allows to say more.
\begin{proposition}
Let $X$ be a smooth toric Fano variety with $c_X=2$. 
Then there is a contraction $g\colon X\to Y_0$, with general fiber a smooth rational curve, where $Y_0$ is a Gorenstein toric Fano variety, with canonical singularities, and having a smooth crepant resolution $Y\to Y_0$ with $\rho_Y=\rho_X-2$.
\end{proposition}
\begin{proof}
First of all we show that there exists a torus-invariant prime divisor $D\subset X$ with $c(D)=2$. 

Let us start with any prime divisor $B\subset X$ with $c(B)=2$, and run a special MMP for $-B$. If the MMP is of type $(a)$, then we get a fixed prime divisor $E\subset X$ such that $c(E)=2$ (see \ref{typea}). Then $E$ must be a torus-invariant divisor, and we set $D:=E$.

Suppose instead that the MMP is of type $(b)$. We keep the same notation as in \ref{typeb}.
Then $X_k$ is smooth and toric, and $\ph\colon X_k\to Y$ is a $\pr^1$-bundle with two disjoint torus-invariant sections $S_1$ and $S_2$. The transforms $\w{S}_i\subset X$ are disjoint torus-invariant divisors, and since  $\w{S}_i\cdot (e+\wi{e})=\w{S}_i\cdot\ell=1$, we can assume that $\w{S}_1\cap E=\emptyset$. Then $\w{S}_1$ is disjoint from $\w{S}_2$ and $E$, thus $c(\w{S}_1)=2$, and we set $D:=\w{S}_1$.

\medskip

Since $D$ is torus-invariant, it is itself a smooth, projective toric variety, and $\Nu(D)$ is generated by the classes of torus-invariant divisors of $D$. These are all given by restriction to $D$ of  torus-invariant divisors of $X$, so that the restriction map $\Nu(X)\to\Nu(D)$ is surjective. Hence
the map $\N(D)\to\N(X)$ is injective, and $\rho_D=\rho_X-2$. 

Let $\alpha$ be the  one-dimensional cone of the the fan of $X$ corresponding to $D$.
It follows from \cite[Lemma 3.3]{fano} that up to replacing $D$ with another torus-invariant prime divisor, we can assume that $-\alpha$ is again  a one-dimensional cone of the fan of $X$; let $D'\subset X$ be the torus-invariant prime divisor corresponding to $-\alpha$. Then $D\cap D'=\emptyset$, and it follows from
\cite[Lemma 3.3 and Th.~3.4(2)]{fano}
that there is a special MMP of type $(b)$ for $-D$:
$$\xymatrix{X  \ar@{-->}[r]^{\sigma}& {X_k}\ar[r]^{\ph}& Y
}$$
where $\ph$ is a $\pr^1$-bundle, $D'$ is contained in the open subset where $\sigma$ is an isomorphism, and $\sigma(D')$ is a section of $\ph$, so that $D'\cong Y$.

We remark that since $X$ is Fano, the divisors $-K_X-D'$ and $-K_X-D-D'$ are nef. Indeed, any extremal ray of $X$ contains the class of some irreducible torus-invariant curve $C\subset X$. It follows from Reid's description of toric Mori theory (see \emph{e.g.} \cite[\S 2 and references therein, in particular Th.~2.3]{fano}) that $B\cdot C\leq 1$ for every torus-invariant prime divisor $B\subset X$, so that $(-K_X-D')\cdot C\geq 0$.
Moreover, 
since $D$ and $D'$ correspond to opposite one-dimensional cones in the fan of $X$, if
$C$ has positive intersection with both $D$ and $D'$, one necessarily has $-K_X\cdot C=2$ and hence $(-K_X-D-D')\cdot C=0$. 
Indeed, if $D\cdot C=D'\cdot C=1$, then the  ``primitive relation''  (introduced by Batyrev) associated to the class $[C]$ is forced to be $v+(-v)=0$, where $v\in\alpha$ is a primitive generator; this yields $-K_X\cdot C=2$.
We refer the reader to \cite[\S 2]{fano} and references therein for more details.

Since $-K_{D'}=(-K_X-D')_{|D'}$ is nef, so is $-K_Y$. Moreover $-K_Y$ is also big, because $Y$ is toric.
Let $\beta\colon Y\to Y_0$ be the birational contraction induced by $-K_Y$, so that $Y_0$ is  a Gorenstein toric Fano variety, with canonical singularities, of dimension $n-1$, and $\beta$ is a crepant resolution.

Consider the  map $g:=\beta\circ\ph\circ\sigma\colon X\dasharrow Y_0$. One can check that $g^*(-K_{Y_0})=-K_X-D-D'$. Since $-K_X-D-D'$ is nef in $X$, $g$ is regular, and we have the statement.
\end{proof}
Besides the toric case, it would be interesting to understand under which conditions, in the setting of Th.~\ref{main}$(i)$, $Y$ has big and nef anticanonical divisor.
\end{parg}
%\bibliographystyle{amsplain}
%\bibliography{BiblioBreve}
\providecommand{\bysame}{\leavevmode\hbox to3em{\hrulefill}\thinspace}
\providecommand{\MR}{\relax\ifhmode\unskip\space\fi MR }
\providecommand{\MRhref}[2]{%
  \href{http://www.ams.org/mathscinet-getitem?mr=#1}{#2}
}
\providecommand{\href}[2]{#2}

\end{document}